\documentclass[12pt]{amsart}
\usepackage{epsfig}
\usepackage{amssymb}
\usepackage{xcolor}
\usepackage{tikz}
\usetikzlibrary{arrows.meta}
\usepackage{float}
\usepackage{graphicx}
\usepackage[pdfstartview=FitH,
            CJKbookmarks=true,
            bookmarksnumbered=true,
            bookmarksopen=true,
            colorlinks=true,
            linkcolor=blue,
            anchorcolor=blue,
            citecolor=red,
            urlcolor=blue
            ]{hyperref}
\newtheorem{thm}{Theorem}[section]
\newtheorem{deff}[thm]{Definition}
\newtheorem{lem}[thm]{Lemma}
\newtheorem{cor}[thm]{Corollary}

\newtheorem{prp}[thm]{Proposition}

\begin{document}

\title[Ulam floating functions]{Ulam floating functions}

\author{Chunyan Liu, Elisabeth M. Werner, Deping Ye, and Ning Zhang}

\address{Chunyan Liu, School of Mathematics and Statistics, Huazhong University of Science and Technology, 1037 Luoyu Road, Wuhan, Hubei 430074, China}

\email{chunyanliu@hust.edu.cn}

\address{Elisabeth M. Werner, Department of Mathematics, Case Western Reserve University, Cleveland, OH 44106, USA}

\email{elisabeth.werner@case.edu}

\address{Deping Ye, Department of Mathematics and Statistics, Memorial University of Newfoundland, St.\ John's, Newfoundland, Canada A1C 5S7 }

\email{deping.ye@mun.ca}

\address{Ning Zhang, School of Mathematics and Statistics, Huazhong University of Science and Technology, 1037 Luoyu Road, Wuhan, Hubei 430074, China}
\email{nzhang2@hust.edu.cn}

\keywords{Affine surface area,  floating body, log-concave functions, Ulam floating function, Ulam floating set.\\ Chunyan Liu and Ning Zhang have been supported by NSF of China (No.
11901217 and No. 11971005).
Elisabeth M. Werner has been supported by NSF grant DMS-2103482.
 Deping Ye  has been supported by an NSERC grant, Canada.}

 \begin{abstract} 
 We extend the notion of Ulam floating sets from convex bodies to Ulam floating functions. We use the  Ulam floating functions to derive a new variational formula 
for the  affine surface area of log-concave functions.

\vskip 2mm \noindent 
{\em 2020 Mathematics Subject Classification:}  52A20, 52A38, 52A41.

\end{abstract}

\maketitle

\section{Introduction}
The study of  affine surface area  was initiated by Blaschke \cite{Blaschke:1923} for smooth convex bodies
in Euclidean space of dimension two and three, and extended to $\mathbb{R}^n$  and general convex bodies by Leichtweiss \cite{Leichtweis:1986},  Lutwak \cite{Lutwak:1996},
and Sch{\"u}tt and Werner \cite{Werner:1990}.
The affine surface area has  remarkable properties. Aside from affine invariance and translation invariance, we
mention only the affine isoperimetric inequality which is a powerful tool in locating extremizers. Therefore  it is not surprising that the 
affine surface area has proved
to be  useful in many problems, e.g.,  Plateau problems
\cite{Trudinger:2000, Trudinger:2002, Trudinger:2008}, the 
approximation theory of convex bodies by polytopes
\cite{ Boroczky:2000, Werner:2018, Werner:2013, Reitzner:2002, Schutt:1991, Schutt:2003} 
and  affine curvature flow \cite{Andrews:1999, Ivaki:2015, Ivaki:2013,  Stancu:2002, Stancu:2003}. 
\par
Blaschke used Dupin's \cite{Dupin:1822} notion of the floating body for his definition of the affine surface area. Dupin's floating body needs not to be convex.
Therefore, Sch{\"u}tt and Werner in \cite{Werner:1990}  used   the (convex) floating body, introduced independently in \cite{Larman:1998, Werner:1990}, 
in their definition of the affine surface area. 
\par
An isomorphic variant of the (convex) floating body is  the metronoid introduced by Huang  and  Slomka \cite{Huang:2019a}.   The metronoid is also called the  Ulam floating body as it is  intimately related to Ulam's long-standing floating body problem which asks whether Euclidean balls are the only convex bodies
that  float in equilibrium in any orientation.  A negative answer to this problem was recently given by  Ryabogin \cite{Ryabogin}. A  close connection between Ulam floating bodies and  the affine surface area was  proved by Huang,  Slomka and Werner \cite{Huang:2019}.
 
In recent years, considerable effort has been devoted to develop a  geometric theory of log-concave functions. A major goal in this area is to extend notions  from convex geometry to a functional setting.  For the functional analogue of the (convex) floating body, this was achieved by  Li,  Sch{\"u}tt and  Werner in \cite{Werner:2019}.  
They introduced  the notion of floating  function and used it to define  an affine surface area for log-concave functions  $f=e^{-\psi}$, $\psi$ convex, as follows \cite{Werner:2019}:
\begin{align}\label{affine-def-f-f}
as(f) =\int_{\mathbb{R}^n}\big(\det(\nabla^2\psi(x))\big)^{\frac{1}{n+2}}e^{-\psi(x)} dx,
\end{align}  
where $\det(\nabla^2\psi(x))$ denotes the determinant of $\nabla^2\psi$, the Hessian of the convex function $\psi$. In Section \ref{asa}, we  explain why it is natural to call this expression  affine surface area of $f$. 
A slightly different definition of the affine surface area for log-concave functions was given  in
\cite{Caglar:2016}.   
 
It is thus  natural to ask  whether the notion of the Ulam floating body can also be extended to a functional setting and whether a connection to the 
 functional affine surface area can be established.  This is carried out in this paper. 
 
In Section \ref{ul-3}, we define the Ulam floating  
functions $\mathsf{M}_{\delta}(\psi)$ and $U_{\delta}(f)=e^{-\mathsf{M}_{\delta}(\psi)}$  for a convex function $\psi$ and  a log-concave function
$f=e^{-\psi}$.
Taking  right-derivatives of the integral difference of a log-concave function  and its Ulam floating function
 gives rise to the affine surface area of the  log-concave function. This is  main result in this paper. 
One of the difficulties when dealing with functions instead of convex bodies is that the convex epigraph of a convex function is not necessarily bounded anymore.
\vskip 2mm
\begin{thm}\label{mainthem2}
	Let $\psi:\mathbb{R}^n\to\mathbb{R}$ be a  convex function such that $
	0<\int_{\mathbb{R}^n}e^{-\psi(x)}dx<\infty$. Then
\vskip 2mm
\noindent  
\begin{align*} 
\lim_{\delta\to0^+} \bigg(\delta^{-\frac{2}{n+2}} \int_{\mathbb{R}^n}(e^{-\psi(x)}-e^{-\mathsf{M}_{\delta}(\psi)(x)})\,dx\bigg)&=  \lim_{\delta\to0^+} \bigg(\delta^{-\frac{2}{n+2}}  \int_{\mathbb{R}^n}\Big|
	\mathsf{M}_{\delta}(\psi)(x)-\psi(x)\Big|e^{-\psi(x)}\,dx\bigg) \nonumber \\&=c_{n+1}\int_{\mathbb{R}^n}\big(\det(\nabla^2\psi(x))\big)^{\frac{1}{n+2}}e^{-\psi(x)}dx, 
\end{align*} where $\,dx$ is the Lebesgue measure on $\mathbb{R}^n$ and
$c_{n+1}$ is a constant given by 
\begin{equation}\label{003}
	c_{n+1}=\frac{n+2}{2(n+4)}\Big(\frac{n+2}{\mathsf{vol}_n(\mathbf{B}_2^n)}\Big)^{\frac{2}{n+2}},
\end{equation} and  $\mathsf{vol}_n(\mathbf{B}_2^n)$ is the volume of $\mathbf{B}_2^n$, the Euclidean unit ball centered at the origin in $\mathbb{R}^n$.    
 \end{thm}
\vskip 4mm
The paper is organized as follows. Section \ref{asa} will provide  background and notations. Ulam floating functions will be introduced in Section \ref{ul-3}, where we also prove some of their basic properties.  The proof of the  relation between  the Ulam floating function and the  affine surface area (i.e., Theorem \ref{mainthem2}),  will be established in Section \ref{ul-4}.

\section{Background and notations} \label{asa}

In this section, we collect  background and notations that will be used throughout the paper. We refer the reader to the books 
\cite{Gardner:2006, Rockafellar:1970, Rockafellar:1998, Schneider:2014}, 
and the articles   \cite{Larman:1998, Huang:2019a, Huang:2019,  Werner:2019,  Werner:1990} for more details.

Let $\mathbb{R}^n$, $n\geq 2$,  be the $n$-dimensional Euclidean space,  $\|x\|$ be the Euclidean norm of $x\in \mathbb{R}^n$, and  $\langle x, y\rangle$ be the inner product of $x, y\in \mathbb{R}^n$. The standard basis of $\mathbb{R}^n$ is denoted by $\{e_1, \cdots, e_n\}$.   By $\partial E$, $E^c$ and $\mathsf{int}(E)$, we mean the  
boundary, complement and interior of  $E\subset \mathbb{R}^n$, respectively. The distance  between $x\in \mathbb{R}^n$ and $E\subset \mathbb{R}^n$ is defined by $$ \mathsf{dist}\big(x, E\big)=\inf_{y\in E}\|x-y\|.$$
Let $\mathbf{B}_2^{n}(x,\rho)$ stand for the closed Euclidean ball in $\mathbb{R}^n$ centered at $x$ with radius $\rho$. In particular, we write in short $\mathbf{B}_2^n=\mathbf{B}_2^n(o,1)$ for the
Euclidean unit ball centered at the origin $o$, and $\mathbb{S}^{n-1}=\partial\mathbf{B}_2^n$ for the unit sphere.  We would like to point out that the notation $o$ always means the origin but its dimension may vary in the later context.  Let  $[x,y)$ denote the ray from $x\in\mathbb{R}^n$ (inclusively) to $y \in\mathbb{R}^n$. 
\par
Let $\mathfrak{C}$ be the collection of  convex functions  $\psi:\mathbb{R}^n\to\mathbb{R}$.
Denote by   $\nabla\psi$ and  $\nabla^2\psi$   the gradient and the  Hessian of $\psi$, respectively. 
 Throughout this paper,  we say that $f: \mathbb{R}^n\to (0, \infty)$ is a
  log-concave function if $f=e^{-\psi}$ with  $\psi\in \mathfrak{C}$. It is well-known that $\nabla\psi$ exists almost everywhere by Rademacher's theorem in \cite{Borwein:2010}, and  $\nabla^2\psi$ exists almost everywhere in $\mathbb{R}^n$   by Alexandrov \cite{Alexandroff:1939} and Busemann and Feller in
\cite{Buse-Feller:1936}. 
\par
A vector $y\in \mathbb{R}^n$ is said to be a subgradient of $\psi$ at $x_0 \in \mathbb{R}^n$ if $\psi(z)-\psi(x_0)\geq \langle y, z-x_0\rangle$ for all $z\in \mathbb{R}^n$. For a convex function $\psi\in \mathfrak{C}$,  the subgradient exists at every point in $\mathbb{R}^n$.  
We say that $\psi\in \mathfrak{C}$ is twice differentiable in a generalized sense at $x_0\in \mathbb{R}^n$ (see e.g.,  \cite{Schutt:2003}), if there exists a linear map $D^2\psi(x_0): \mathbb{R}^n\to \mathbb{R}^n$ such that, for all subgradients $\partial \psi$ of $\psi$,
\begin{align} \|\partial \psi(x)-\partial\psi(x_0)-\nabla^2\psi(x_0)(x-x_0)\| \leq \omega(\|x-x_0\|) \|x-x_0\|, \label{dupin-1}\end{align}  holds in a neighborhood $\mathcal{U}(x_0)\subseteq \mathbb{R}^n$,  where $\omega(\cdot): (0, \infty)\rightarrow (0, \infty)$ is a properly  chosen function with $\lim_{t\to 0^+} \omega(t)=0$. In this case, we call $\nabla^2\psi(x_0)$ the {\em generalized Hessian matrix} of $\psi$ at $x_0$. 
\par
Let $\theta\in \mathbb{S}^{n-1}$ and $a\in\mathbb{R}$. Then $$\mathsf{H}(\theta,a)=\{x\in\mathbb{R}^n:\langle x,\theta \rangle=a\} $$  is a hyperplane with normal vector $\theta$. The hyperplane  $\mathsf{H}(\theta,a)$ defines  two 
 closed half spaces   
\[
\mathsf{H}^{+}(\theta,a)=\{x\in\mathbb{R}^n:\langle x,\theta \rangle\ge a\} \ \ \mathrm{and} \ \ 
\mathsf{H}^{-}(\theta,a)=\{x\in\mathbb{R}^n:\langle x,\theta\rangle\leq a\}.
\] We often  write $\mathsf{H}$, $\mathsf{H}^{+},$ and $\mathsf{H}^{-}$ instead of $\mathsf{H}(\theta, a)$, $\mathsf{H}^{+}(\theta, a),$ and $\mathsf{H}^{-}(\theta, a)$ if no confusion occurs. 
  
 A subset  $K$ of  $\mathbb{R}^n $ is a convex body if  $K$ is a compact convex set with non-empty interior. The support function of a convex body $K$, $h_K:  \mathbb{S}^{n-1} \rightarrow \mathbb{R}$, is defined by 
$$h_{K}(\theta)=\max_{x\in K} \langle x,\theta \rangle  \   \ \mathrm{for}\ \ \ \theta\in\mathbb{S}^{n-1}.$$    Note that any convex body $K$ is  uniquely determined by its support function. By $N_{K}(z)$, $\kappa_{K}(z)$,  and $\mu_{\partial K}$,  we mean the unit outer normal at $z\in\partial K$,   
the (generalized) Gauss curvature at
$z\in\partial K$, and  the surface area measure on $\partial K$, respectively. Also $\mathsf{vol}_n(E)$ stands for the $n$-dimensional volume of $E\subset \mathbb{R}^n$. 
 
An important notion in affine,  convex and differential geometry is the {\em affine surface area}, which was  introduced by Blaschke \cite{Blaschke:1923}
in dimensions $2$ and $3$ for smooth enough convex bodies.
For a  convex body  $K\subset \mathbb{R}^n$, it is defined as    
  \begin{equation}\label{02}
	as(K)=\int_{\partial K}(\kappa_{K}(z))^{\frac{1}{n+1}}d\mu_{\partial K}(z). 
\end{equation}  
\vskip 2mm
It was shown in \cite{Werner:1990} that the affine surface area can be obtained by a first order variation of the volume of $K$ via a family of convex floating bodies. Let $K$ be a convex body in $\mathbb{R}^n$ and $0<\delta<\mathsf{vol}_n(K)$ be small enough. The  convex floating body $K_{\delta}$ of $K$ is a variant of Dupin's floating body \cite{Dupin:1822} and was introduced in \cite{Larman:1998, Werner:1990} as  the intersection of all
halfspaces $\mathsf{H}^{+}$ whose defining hyperplanes $\mathsf{H}$ cut off sets of volume $\delta$ from $K$, i.e.,   
\begin{align}\label{convex-f}
K_{\delta}=\bigcap_{\{\mathsf{H}:\ \mathsf{vol}_n(\mathsf{H}^{-}\cap K)=\delta\}}\mathsf{H}^{+}.
\end{align}
It has been proved in \cite{Werner:1990} that \begin{equation} as(K)=d_n^{-1} \lim_{\delta\rightarrow 0^+} \frac{\mathsf{vol}_n(K)-\mathsf{vol}_n(K_\delta)}{\delta^{\frac{2}{n+1}}}, \label{affine-1} \end{equation}  where $d_n$ is the constant given by $$d_n=\frac{1}{2} \Big(\frac{n+1}{\mathsf{vol}_{n-1}(\mathbf{B}_2^{n-1})}\Big)^{\frac{2}{n+1}}.$$
\vskip 3mm
Let $K\subset \mathbb{R}^{n}$ be a convex body with $o\in\partial K$ and $N_{K}(o)=-e_{n}$. The boundary of $K$ (around the origin)  can be represented  by $x_n=\varphi(x_1, \cdots, x_{n-1})$ for some convex function $\varphi: \mathbb{R}^{n-1}\to [0, \infty)$ such that $\varphi(o)=0$.   The indicatrix of Dupin (see e.g., \cite{Schutt:2003})  is the quadratic form $$\Big\{ y\in \mathbb{R}^{n-1}: \ \big\langle y, \big(\nabla^2\varphi(o)\big) y\big\rangle=1\Big\},$$
where $\nabla^2\psi(o)$ is the generalized Hessian matrix of $\psi$ at $o$ given by \eqref{dupin-1}.   
The following lemma  was proved in  \cite[Lemma 11 (ii)]{Werner:1990}.   
\vskip 2mm
\begin{lem} \cite{Werner:1990} \label{lemma2}
	Let $K\subset \mathbb{R}^{n}$ be a convex body with $o\in\partial K$ and $N_{K}(o)=-e_{n}$. Suppose  that the indicatrix of Dupin at the origin $o$ exists and is an $(n-1)$ dimensional sphere with radius $\sqrt{\rho}$. Let $\xi$ be an interior point of $K$.  
Then there is  $\varepsilon_0>0$ such  that,  for all $z\in [o, \xi)$ with $\|z\|<\varepsilon_0$,  
	\[ \mathsf{dist}\big(z, \mathbf{B}_2^{n}(\rho e_n,\rho)^c\big)\leq z_n \leq \mathsf{dist}\big(z, \mathbf{B}_2^{n}(\rho e_n,\rho)^c\big) \bigg(1+\frac{2 \mathsf{dist}\big(z, \mathbf{B}_2^{n}(\rho e_n,\rho)^c\big) }{\langle \frac{\xi}
		{\|\xi\|},N_{K}(o)\rangle^2\rho }\bigg).
	\]  \end{lem}
\vskip 3mm
The identity  \eqref{affine-1} shows that $as(K)$ can be expressed as a derivative of volume, using floating bodies. That is just one example of this
phenomenon and in fact floating bodies can be replaced by other families of bodies constructed from $K$. We refer to e.g., \cite{MeyerWerner:1998, SchuettWerner:2004, Werner:1994, Werner:1999,  WernerYe:2008, WernerYe:2010}
and only mention in more detail the {\em Ulam floating body (or metronoid)}  $\mathsf{M}_{\delta}(K)$. These bodies were introduced in \cite{Huang:2019a} (see also   \cite{Huang:2019}), by \[
\mathsf{M}_{\delta}(K):=\mathsf{M}(\nu_{K, \delta})=
\bigcup_{g \in \mathcal{F}_K}
\Big\{\int_{\mathbb{R}^n}yg(y)d\nu_{K,  \delta}(y)\Big\},
\] where  $d\nu_{K,  \delta}=\delta^{-1}\mathbf{1}_{K}dx$ with $\mathbf{1}_{K}$ being the 
characteristic function of $K$ and $\,dx$ the Lebesgue measure on $\mathbb{R}^n$, and 
$$\mathcal{F}_K=\Big\{g: \mathbb{R}^n\to[0,1]: \int_{\mathbb{R}^n}g\,d\nu_{K,  \delta}=1\  \mathrm{and}\ \int_{\mathbb{R}^n} y g(y)d\nu_{K,  \delta}(y) \ \mathrm{ exists}\Big\}.$$ 
\par
It has been noted in \cite{Huang:2019a} that  $\mathsf{M}_{\delta}(K)$ is  convex and  in  \cite[Theorem 1.1]{Huang:2019}  that for  $0<\delta<\mathsf{vol}_n(K)$ small enough, $\mathsf{M}_{\delta}(K)$ is a isomorphic to $K_{\delta}$ in the sense that
\[
K_{(1-\frac{1}{e})\delta}\subseteq\mathsf{M}_{\delta}(K)\subseteq K_{\frac{1}{e}\delta}.
\] 
As shown in \cite[Proposition 2.1]{Huang:2019a}, the support function of $\mathsf{M}_{\delta}(K)$ can be calculated by 
$h_{\mathsf{M}_{\delta}(K)}(\theta)=\langle z_\theta,\theta\rangle$ for $\theta\in \mathbb{S}^{n-1}$, where  \[
z_\theta=\frac{1}{\delta}\int_{\{x\in K: \langle x,\theta\rangle \geq R_{\delta}(\theta)\}}x\,dx=\frac{1}{\delta}\int_{K\cap \mathsf{H}^+(\theta, R_{\delta}(\theta))}x\,dx
\] is the  barycentre  of $K\cap  \mathsf{H}^+(\theta, R_{\delta}(\theta))$ and $R_{\delta}(\theta)\in \mathbb{R}$ satisfies that $
\mathsf{vol}_n(K\cap \mathsf{H}^+(\theta, R_{\delta}(\theta)))=\delta. $ 
\vskip 2mm
Moreover, it was shown in \cite{Huang:2019} that  a variational argument involving the Ulam floating bodies leads to the affine surface area, namely
 \begin{equation}  \lim_{\delta\rightarrow 0^+} \frac{\mathsf{vol}_n(K)-\mathsf{vol}_n(\mathsf{M}_\delta(K))}{\delta^{\frac{2}{n+1}}}=c_n as(K), \label{affine-1-M} \end{equation}  where $c_n$ is the constant given by $$ 
	c_{n}=\frac{n+1}{2(n+3)}\Big(\frac{n+1}{\mathsf{vol}_{n-1}(\mathbf{B}_2^{n-1})}\Big)^{\frac{2}{n+1}}.$$  
\vskip 3mm
We will make use of the following lemma which has been proved in \cite[Lemma 3.5]{Huang:2019}. 

 \begin{lem}\cite{Huang:2019} \label{lem: UnaffectedCut} 
	Let $K\subseteq\mathbb{R}^{n}$ be a convex body. Assume that $o\in \partial K$ and $N_K(o)=-e_n$ is the unique outer normal vector to $\partial K$ at $o$.  Then, for each $t>0$,  there exists $r>0$ such that for any $\delta>0$,
	\[
	\mathsf{M}_{\delta}(K) \cap\mathbf{B}_2^{n}(o,r)=\mathsf{M}_{\delta}\big(K\cap \mathsf{H}^+(-e_n, -t)\big) \cap\mathbf{B}_2^{n}(o,r)=\mathsf{M}_{\delta}\big(K\cap \mathsf{H}^{-}(e_n, t)\big) \cap\mathbf{B}_2^{n}(o,r).
	\]
\end{lem}
\vskip 4mm
For $\psi\in \mathfrak{C}$, 
let $\mathsf{epi}(\psi)$ stand for  the epigraph of $\psi$, namely, \[
\mathsf{epi}(\psi)=\big\{(x,y)\in\mathbb{R}^n\times\mathbb{R}: y\ge\psi(x)\big\}.
\] Clearly, $\mathsf{epi}(\psi)$ is a closed convex set in $\mathbb{R}^{n+1}$. For each $x\in \mathbb{R}^n$, let $z_x=(x,\psi(x))$ be the point on the boundary of $\mathsf{epi}(\psi)$. 
It is well-known that the Gauss curvature $\kappa_{\psi}(z_x)$ and the outer unit normal $N_{\psi}(z_x)$ of $\partial\mathsf{epi}(\psi)$ at  $z_x$ are given by (see, e.g., \cite{Caglar:2016}), 
\begin{align}\label{qu1}
	\kappa_{\psi}(z_x)=\frac{\det(\nabla^2\psi(x))}{(1+\|\nabla\psi(x)\|^2)^\frac{n+2}{2}}\ \ \mathrm{and}\ \ \ N_{\psi}(z_x)=\frac{(\nabla \psi(x), -1)}{(1+\|\nabla\psi(x)\|^2)^\frac{1}{2}}.\end{align} Moreover, the following formula holds: \begin{align}\label{var-11}
(1+\|\nabla\psi(x)\|^2)^\frac{1}{2}dx=d\mu_{\partial\mathsf{epi}(\psi)}(z_x). \end{align}
\vskip 3mm	
	Let  $f=e^{-\psi}$ be   a log-concave function.   It is also well known (see e.g., \cite{Rockafellar:1998})  that  $\int e^{-\psi}\,dx <\infty$ holds if and only if there exist  constants $a>0$ and $b\in \mathbb{R}$ such that $\psi(x) \geq a \|x\|+b$ for all $x\in \mathbb{R}^n$. As remarked above, the quantity 	 
\begin{equation}\label{00}
	as(f)=\int_{\mathbb{R}^n}(\det(\nabla^2\psi(x)))^{\frac{1}{n+2}}e^{-\psi(x)}dx,
\end{equation} 
is called the affine surface area of the log-concave function $f=e^{-\psi}$. This is justified  as it shares many properties with the affine surface area for convex bodies. 
For one, this quantity is an affine invariant. We have for all affine transformations
$T: \mathbb{R}^n\to\mathbb{R}^n$ such that  $\det T$, the determinant of $T$,  is nonzero,  
\[
as(f\circ T)=|\det T|^{-\frac{n}{n+2}}as(f).
\] 
Then, this quantity is a valuation, namely, 
\[
as(f_1)+as(f_2)=as(\max(f_1,f_2))+as(\min(f_1,f_2)),
\] 
holds  for log-concave functions $f_1=e^{-\psi_1}$ and  $f_2=e^{-\psi_2}$ such that $\min(\psi_1,\psi_2)$ is convex.  
\par
\noindent
Moreover, it follows from  \eqref{qu1} and \eqref{var-11} that  
\begin{equation}\label{01}
	as(f)=\int_{\partial\mathsf{epi}(\psi)}(\kappa_{\mathsf{epi}(\psi)}(z_x))^\frac{1}{n+2}~e^{-\langle z_x,e_{n+1}\rangle}d\mu_{\partial\mathsf{epi}(\psi)}(z_x),
\end{equation} 
and the expression for $as(f)$ in \eqref{01} resembles the one of the affine surface area $as(K)$ of  a convex body $K\subseteq\mathbb{R}^{n+1}$ when the bounded convex body $K$ is replaced by the unbounded convex set $\mathsf{epi}(\psi)$. Readers are referred to \cite{Werner:2019} for more details of the affine surface area $as(f)$ of log-concave functions $f$. We also refer to \cite{Caglar:2016} for a slightly different definition of the affine surface area of a log-concave function and the related properties.

Motivated by the convex floating body, Li,   Sch\"{u}tt and Werner in \cite{Werner:2019} introduced the floating function for a log-concave function. 
For a log concave function $f=e^{-\psi}$  and  $\delta>0$ they first defined the floating set of $\mathsf{epi}(\psi)$ by \begin{align} \label{floating-epi} (\mathsf{epi}(\psi))_{\delta} =\bigcap_{\{\mathsf{H}:\ \mathsf{vol}_{n+1}(\mathsf{H}^{-}\cap \mathsf{epi}(\psi))\leq \delta\}}\mathsf{H}^{+}.\end{align} Then they defined the floating function $\psi_{\delta}$ of $\psi$  via  its epigraph as follows
\begin{align} \label{def-floating-f-1} \mathsf{epi}(\psi_{\delta})=(\mathsf{epi}(\psi))_{\delta}.\end{align}  
And finally,  the floating function $f_{\delta}$ of  $f=e^{-\psi}$ is given by $f_{\delta}=e^{-\psi_{\delta}}$. 
\par
It was proved in \cite[Theorem 1]{Werner:2019} that \begin{align} as(f)&=\int_{\mathbb{R}^n}(\det(\nabla^2\psi(x)))^{\frac{1}{n+2}}e^{-\psi(x)}dx \nonumber \\ &=
\lim_{\delta\to0^+}  \bigg(\big(d_{n+1} \cdot \delta^{\frac{2}{n+2}}\big)^{-1} \int_{\mathbb{R}^n}(e^{-\psi(x)}-e^{- \psi_{\delta}(x)})\,dx\bigg), \label{affine-int-f}
\end{align}  where the constant $d_{n+1}$ is given by  
\begin{equation}\label{003-d}
	d_{n+1}=\frac{1}{2}\Big(\frac{n+2}{\mathsf{vol}_n(\mathbf{B}_2^n)}\Big)^{\frac{2}{n+2}}.
\end{equation}
 \vskip 3mm
 We shall need the following  lemma, whose proof  can be found in e.g., \cite{Werner:2019,Schutt:2003}.

\begin{lem}\label{lemma1}
	Let $r>0$ be a constant and $\mathcal{E} \subseteq \mathbb{R}^{n+1}$ be an ellipsoid given by 
	\[
	\mathcal{E}=\Big\{z \in\mathbb{R}^{n+1}:  \Big(\frac{z_{n+1}-r}{r}\Big)^2+\sum_{i=1}^{n}\Big(\frac{z_i}{r}\Big)^2 \leq1\Big\}.
	\]
Then, for all $h\in (0, 2r)$, one has \begin{align}\label{ineq-12-23-1}
	 \Big(1-\frac{h}{2r}\Big)^{\frac{n}{n+2}}h \leq d_{n+1}\cdot \left(\frac{\mathsf{vol}_{n+1}(\mathcal{E}\cap\mathsf{H}^{-}(e_{n+1},h))}{r^{\frac{n}{2}}}\right)^{\frac{2}{n+2}} \leq h.
	\end{align} 	
\end{lem} 
\vskip 7mm

\section{Ulam floating functions}\label{ul-3}
Motivated by the Ulam floating bodies, we will now introduce the Ulam floating functions. We  define first the Ulam floating set
$\mathsf{M}_{\delta}(C)$ for a convex, closed, not necessarily bounded set $C\subseteq\mathbb{R}^{n+1}$.   For convenience, we denote  $d\nu_{C,  \delta} =\delta^{-1}\mathbf{1}_{C}dx$.   
Let  $$\mathcal{F}_C=\Big\{g: \mathbb{R}^{n+1}\to[0,1]: \int_{\mathbb{R}^{n+1}}g\,d\nu_{C,  \delta}=1\  \mathrm{and}\ \int_{\mathbb{R}^{n+1}} yg(y)d\nu_{C,  \delta}(y) \ \mathrm{ exists}\Big\}.$$
\vskip 3mm
\begin{deff}\label{def-M-CC} (Ulam floating set)
	Let $C$ be a closed convex set in $\mathbb{R}^{n+1}$ with nonempty interior. For $\delta>0$,  the
	Ulam floating set of $C$ is defined by 
	\[
	\mathsf{M}_{\delta}(C)=\mathsf{M}(\nu_{C,  \delta})=
	\bigcup_{g\in \mathcal{F}_C}
	\Big\{\int_{\mathbb{R}^{n+1}}y g(y)d\nu_{C,  \delta}(y)\Big\}. 
	\] 
\end{deff} 
Clearly, the Ulam floating set $\mathsf{M}_{\delta}(C)$ is  closed  and convex.    As we are mainly interested in the Ulam floating functions,  from now on, our discussion will be concentrated on the case when $C=\mathsf{epi}(\psi)$ with $\psi\in \mathfrak{C}$. As $\mathsf{epi}(\psi)$ is a closed convex set in $\mathbb{R}^{n+1}$ with nonempty interior, we define  \begin{equation} \label{f-17}
	\mathsf{M}_{\delta}(\mathsf{epi}(\psi))=
	\mathsf{M}(\nu_{\mathsf{epi}(\psi), \delta})=
	\bigcup_{g \in \mathcal{F}_{\mathsf{epi}(\psi)}}  
	\Big\{\int_{\mathbb{R}^{n+1}}y g(y)d\nu_{\mathsf{epi}(\psi), \delta}(y)\Big\},
\end{equation} where the set $\mathcal{F}_{\mathsf{epi}(\psi)}$ is given by 
$$\mathcal{F}_{\mathsf{epi}(\psi)}=\Big\{g: \mathbb{R}^{n+1}\to[0,1]: \int_{\mathbb{R}^{n+1}}g\,d\nu_{\mathsf{epi}(\psi), \delta}=1\  \mathrm{and}\ \int_{\mathbb{R}^{n+1}} y g(y)\,d\nu_{\mathsf{epi}(\psi), \delta}(y) \ \mathrm{ exists}\Big\}.$$ Since  $\mathsf{M}_{\delta}(\mathsf{epi}(\psi))$ is  closed  and convex, there is  a unique convex function
$\mathsf{M}_{\delta}(\psi):\mathbb{R}^n\to\mathbb{R}$ such that
\[
\mathsf{M}_{\delta}(\mathsf{epi}(\psi))=\mathsf{epi}(\mathsf{M}_{\delta}(\psi)).
\]
This leads to the following definitions of the Ulam floating functions for convex and log-concave functions.
\vskip 3mm
\begin{deff}
	Let $\psi\in \mathfrak{C}$ and  $\mathsf{epi}(\psi)$ be its epigraph. Let
	$\delta\in (0, \infty)$   and  $f=e^{-\psi}$ be a log-concave function. 
	\vskip 1mm 
	\noindent \textbf{(i)} The Ulam floating function   $\mathsf{M}_{\delta}(\psi)$ of $\psi$ is defined  by 
	\[
	\mathsf{M}_{\delta}(\mathsf{epi}(\psi))=\mathsf{epi}(\mathsf{M}_{\delta}(\psi)).
	\]
	
	\noindent \textbf{(ii)}  The Ulam floating function $U_{\delta}(f)$ of $f$ is 
	defined as 
	\[
	U_{\delta}(f)=e^{-\mathsf{M}_{\delta}(\psi)}.
	\]
\end{deff}  
\vskip 2mm
As $\mathsf{M}_{\delta}(\mathsf{epi}(\psi))\subseteq  \mathsf{epi}(\psi)$ for all $\delta>0$, one sees that $\mathsf{epi}(\mathsf{M}_{\delta}(\psi)) \subseteq  \mathsf{epi}(\psi)$ which further yields  $\mathsf{M}_{\delta}(\psi)\geq \psi$. More generally, if $0<\delta_1<\delta_2$, then $\mathsf{M}_{\delta_2}(\mathsf{epi}(\psi))\subseteq \mathsf{M}_{\delta_1}(\mathsf{epi}(\psi))$.  
 Consequently,  $\mathsf{M}_{\delta_2}(\psi)\geq \mathsf{M}_{\delta_1}(\psi)$ and $U_{\delta_2} (f) \leq U_{\delta_1}(f)\leq f$ for $0<\delta_1<\delta_2$. It can also be checked easily  that for all $x\in \mathbb{R}^n$, $$\lim_{\delta\rightarrow 0^+} \mathsf{M}_{\delta}(\psi)=\psi \ \ \mathrm{and} \ \ \lim_{\delta\rightarrow 0^+} U_{\delta}(f)=f.$$
\vskip 3mm
Following the discussion in \cite{Huang:2019} for properties of the Ulam floating bodies,  we now describe some basic properties for the Ulam floating set $\mathsf{M}_{\delta}(\mathsf{epi}(\psi))$.  
\newline
Let the function $\eta:  \mathbb{S}^{n}\times\mathbb{R}\to[0,\infty]$ be defined by  
\begin{align}\label{def-eta} 
	 \eta(\theta,a)=\mathsf{vol}_{n+1}(\mathsf{epi}(\psi)\cap\mathsf{H}^{-}(\theta,a)), \ \ \mathrm{for} \ \theta\in   \mathbb{S}^{n} \ \mathrm{and} \ a\in \mathbb{R}. 
\end{align} Let  $\eta_{\theta}(\cdot)=\eta(\theta, \cdot)$ for any fixed $\theta\in \mathbb{S}^{n}$ and let $$\Sigma_{\psi}=\big\{\theta\in   \mathbb{S}^{n}: \eta_{\theta}(a)<\infty \ \ \mathrm{for \ all}\ \  a\in \mathbb{R}\big\}.$$
Note that if $\theta\in \Sigma_{\psi}$, then $\eta_{\theta}: [-h_{\mathsf{epi}(\psi)}(-\theta), \infty)\rightarrow [0, \infty)$ is continuous and strictly increasing. 
Indeed,  it is obvious that $$\mathsf{epi}(\psi)\cap\mathsf{H}^{-}(\theta,a)\subset\mathsf{epi}(\psi)\cap\mathsf{H}^{-}(\theta,b)$$ for $-h_{\mathsf{epi}(\psi)}(-\theta)<a<b<\infty$,  and then 
	$\eta_{\theta}(a)<\eta_{\theta}(b)$. Let $a_0\in [-h_{\mathsf{epi}(\psi)}(-\theta),\infty)$ and  $\zeta\in (0, 1)$. Then, whenever $|a-a_0|\leq\zeta$ and $a\in [-h_{\mathsf{epi}(\psi)}(-\theta),\infty)$, one has 
	\[ 
	\mathsf{epi}(\psi)\cap\mathsf{H}^{-}(\theta,a_0-\zeta)\subseteq
\mathsf{epi}(\psi)\cap\mathsf{H}^{-}(\theta,a)\subseteq
	\mathsf{epi}(\psi)\cap\mathsf{H}^{-}(\theta,a_0+\zeta).
	\]  This further gives 
	\begin{align*}
	|\eta_{\theta}(a)-\eta_{\theta}(a_0)| \leq\int_{a_0-\zeta}^{a_0+\zeta} \mathsf{vol}_n(\{x\in\mathsf{epi}(\psi):\left<x,\theta\right>=t\})
\,dt \leq  2D\zeta,  
	\end{align*} where $D<\infty$ is given by  
	$$D= \max \Big\{\mathsf{vol}_n(\mathsf{epi}(\psi)\cap \mathsf{H}(\theta, t)): t\in[a_0-1, a_0+1]\cap[-h_{\mathsf{epi}(\psi)}(-\theta),\infty)\Big\}.$$ 	
	Hence, for all $\zeta\in (0, \frac{\varepsilon}{2D})$, one gets 
$$
|\eta_{\theta}(a)-\eta_{\theta}(a_0)|<\varepsilon, \ \ \mathrm{for \ all} \ \ a\in (a_0-\zeta, a_0+\zeta)\cap [-h_{\mathsf{epi}(\psi)}(-\theta),\infty).
$$
This shows that  $\eta_{\theta}(\cdot)$ is continuous on $[-h_{\mathsf{epi}(\psi)}(-\theta),\infty)$.   
 Consequently,  the inverse function of $\eta_{\theta}$  exists and will be denoted by $\eta^{-1}_{\theta}(\cdot):  [0, \infty)\rightarrow [-h_{\mathsf{epi}(\psi)}(-\theta), \infty)$. Clearly,  $\eta^{-1}_{\theta}(\cdot)$ is also  
continuous  and strictly increasing. Let $\delta>0$ be fixed. Then, 
\begin{align*} 
\eta^{-1}_{\theta}(\delta)&=\min\Big\{a \in\mathbb{R}: \ \nu_{\mathsf{epi}(\psi), \delta} \big(\big\{x
\in\mathbb{R}^{n+1}: \langle x,\theta\rangle \leq  a \big\}\big)\ge1\Big\}\\ &=\min\big\{a \in\mathbb{R}: \ \mathsf{vol}_{n+1}(\mathsf{epi}(\psi)\cap\mathsf{H}^{-}(\theta, a)) \geq \delta \big\}\\ &= \min\big\{a \in\mathbb{R}: \ \eta_{\theta}(a) \geq \delta \big\}.
\end{align*} 
Moreover for $\delta>0$,  
$$\eta_{\theta}(\eta^{-1}_{\theta}(\delta))= \mathsf{vol}_{n+1}(\mathsf{epi}(\psi)\cap\mathsf{H}^{-}(\theta, \eta^{-1}_{\theta}(\delta)))=\delta.
$$ 
Define
$g_{\theta}:\mathbb{R}\to[0,1]$ as follows:  \[
g_{\theta}(t) =
\begin{cases} 
	0,  & t\geq \eta^{-1}_{\theta}(\delta) \\
	1,  & t<\eta^{-1}_{\theta}(\delta). 	
\end{cases}
\]
Clearly,  $0\leq g_{\theta}\leq 1$ and \begin{align} \int_{\mathbb{R}^{n+1}}g_{\theta}
(\langle x,\theta\rangle)d\nu_{\mathsf{epi}(\psi),\delta}(x)&= \frac{1}{\delta} \int _{\{x\in \mathsf{epi}(\psi): \langle x, \theta\rangle< \eta^{-1}_{\theta}(\delta)\}}\,dx \nonumber \\ & = \frac{\mathsf{vol}_{n+1}(\mathsf{epi}(\psi)\cap\mathsf{H}^{-}(\theta, \eta^{-1}_{\theta}(\delta)))}{\delta} =1.  \label{condition-g-theta-22-01}  \end{align}  
Note that  $g_{\theta}(\langle \theta, \cdot\rangle)\in \mathcal{F}_{\mathsf{epi}(\psi)}$  because $y_\theta$ exists (and indeed $y_\theta \in\mathsf{M}_{\delta}(\mathsf{epi}(\psi))$), where
\begin{eqnarray} 
y_\theta &=&\int_{\mathbb{R}^{n+1}}x g_{\theta}(\langle x,\theta\rangle)d\nu_{\mathsf{epi}(\psi),\delta}(x)  =\frac{1}{\delta} \int _{\{x\in \mathsf{epi}(\psi): \langle x, \theta\rangle< \eta^{-1}_{\theta}(\delta)\}} x \,dx \nonumber \\&
=&\frac{1}{\mathsf{vol}_{n+1}(\mathsf{epi}(\psi)\cap\mathsf{H}^{-}(\theta, \eta^{-1}_{\theta}(\delta)))} \int _{\mathsf{epi}(\psi)\cap\mathsf{H}^{-}(\theta, \eta^{-1}_{\theta}(\delta))} x \,dx. \label{def-y-theta} 
\end{eqnarray} 
In other words,  $y_{\theta}$ is the barycenter of
$\mathsf{epi}(\psi)\cap\mathsf{H}^{-}(\theta, \eta^{-1}_{\theta}(\delta))$, which is illustrated in Figure \ref{figure-1}.  \begin{figure}[htpb] 
	\centering
	\begin{tikzpicture}[node distance=2cm]
	\draw[->] (-5,0) -- (5,0) node[below] {$\mathbb{R}^{n}$};
	\draw[->] (0,-0.5) -- (0,5.5) node[left] {$\mathbb{R}$};
	\draw[-] (-0.3,0) -- (-0.3,0)node[below]{$o$};
	\draw[domain=-3:3] plot(\x, 1/2*\x*\x+0.5)node[right]{$\mathsf{epi}(\psi)$};
	\draw[dashed,->] (0,0)--(2.5,5) node[above] {$\theta$};
	\draw[dashed,-] (-4,4.5)--(4,0.5) node[above] {$\mathsf{H}\big(\theta,\eta^{-1}_{\theta}(\delta)\big)$};
    \draw[dashed,-] (-4.5,15/4)--(3.5,-0.25) node[left]{$\mathsf{H}$};
    \fill[domain=-2:1,fill=gray] (-2,5/2) plot(\x,1/2*\x*\x+0.5)--(1,1) plot(\x-1.5,-1/2*\x+1.5) node[left]{$\big(\mathsf{epi}(\psi)\big)\cap\mathsf{H}^{-}$};
    \filldraw[fill=black] (-0.5,1.75) circle (0.08) node[above] {$y_{\theta}$};
	\end{tikzpicture} 
	\caption{It illustrates that $y_{\theta}$ is a boundary point of $\mathsf{M}_{\delta}(\mathsf{epi}(\psi))$. 
The halfspace $\mathsf{H}^{-}\big(\theta,\eta^{-1}_{\theta}(\delta)\big)$, defined by 	
	the support hyperplane $\mathsf{H}=\mathsf{H}\big(\!\!-\theta, h_{\mathsf{M}_{\delta}(\mathsf{epi}(\psi))}(-\theta)\big)$ to $\mathsf{M}_{\delta}(\mathsf{epi}(\psi))$ in the direction of $(-\theta)$, cuts off a  set of volume $\delta$ from $\mathsf{epi}(\psi)$.}  \label{figure-1}
\end{figure}
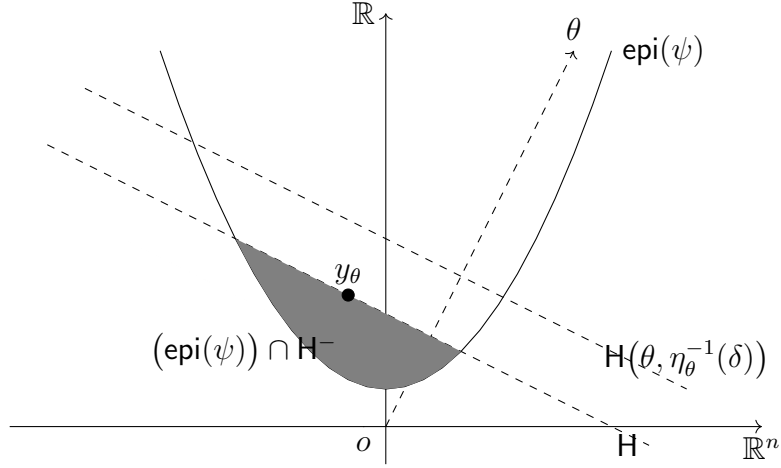 
\vskip 2mm We now prove the following result. 

\begin{prp}\label{pro1} Let $\psi\in \mathfrak{C}$. For $\delta>0$ and $\theta\in \Sigma_{\psi}$, one has $y_{\theta}\in \partial \big(\mathsf{M}_{\delta}(\mathsf{epi}(\psi))\big)$ and  \begin{align} h_{\mathsf{M}_{\delta}(\mathsf{epi}(\psi))}(-\theta)=\langle y_{\theta},-\theta \rangle.\label{supp-22-02}\end{align} 
\end{prp}
\begin{proof} 
Let $\delta>0$ and $\theta\in \Sigma_{\psi}$. Clearly, $y_{\theta}\in \mathsf{M}_{\delta}(\mathsf{epi}(\psi))$. The desired result in Proposition \ref{pro1}  follows immediately once \eqref{supp-22-02} is verified.  Let $y\in \mathsf{M}_{\delta}(\mathsf{epi}(\psi))$. 
There exists a function 
$g\in \mathcal{F}_{\mathsf{epi}(\psi)}$,  $g: \mathbb{R}^{n+1}\to[0,1]$ such that
\begin{align} 1 & = \int_{\mathbb{R}^{n+1}}g\,d\nu_{\mathsf{epi}(\psi), \delta}=\frac{1}{\delta} \int_{\mathsf{epi}(\psi)}g\,dx, \label{equation-22-02} \\  y&= \int_{\mathbb{R}^{n+1}} x g(x)\,d\nu_{\mathsf{epi}(\psi), \delta}(x)=\frac{1}{\delta} \int_{\mathsf{epi}(\psi)} x g(x)\,dx.  \label{equation-22-02-1} \end{align}
 These, together with \eqref{condition-g-theta-22-01}, further imply that \begin{align*}   0&=\int_{\mathsf{epi}(\psi)}g\,dx- \int_{\mathsf{epi}(\psi)}g_{\theta}(\langle x, \theta\rangle) \,dx \\&=
 \int _{\{x\in \mathsf{epi}(\psi): \langle x, \theta\rangle\geq \eta^{-1}_{\theta}(\delta)\}} g(x) \,dx + \int _{\{x\in \mathsf{epi}(\psi): \langle x, \theta\rangle< \eta^{-1}_{\theta}(\delta)\}} \big(g(x)-1\big)\,dx.  \end{align*}  
 Dividing both sides by $\eta_{\theta}(\delta)$, it follows from  \eqref{def-y-theta},  \eqref{equation-22-02}, \eqref{equation-22-02-1} and  $0\leq g\leq 1$  that  
 \begin{align*}   0& = 
 \int _{\{x\in \mathsf{epi}(\psi): \langle x, \theta\rangle\geq \eta^{-1}_{\theta}(\delta)\}} g(x) \eta_{\theta}^{-1}(\delta)  \,dx + \int _{\{x\in \mathsf{epi}(\psi): \langle x, \theta\rangle< \eta^{-1}_{\theta}(\delta)\}} \big(g(x)-1\big)\eta_{\theta}^{-1}(\delta)  \,dx\\ & \leq 
 \int _{\{x\in \mathsf{epi}(\psi): \langle x, \theta\rangle\geq \eta^{-1}_{\theta}(\delta)\}} g(x) \langle x, \theta\rangle \,dx + \int _{\{x\in \mathsf{epi}(\psi): \langle x, \theta\rangle< \eta^{-1}_{\theta}(\delta)\}} \big(g(x)-1\big)  \langle x, \theta\rangle  \,dx \\ & = 
 \int _{\mathsf{epi}(\psi)} g(x) \langle x, \theta\rangle \,dx - \int _{\{x\in \mathsf{epi}(\psi): \langle x, \theta\rangle< \eta^{-1}_{\theta}(\delta)\}}    \langle x, \theta\rangle  \,dx \\&=\delta \cdot \big(\langle y, \theta\rangle-\langle y_{\theta}, \theta\rangle\big).  \end{align*} This gives $\langle y_{\theta}, -\theta\rangle \geq \langle y, -\theta\rangle$ for all $y\in \mathsf{M}_{\delta}(\mathsf{epi}(\psi))$, as desired. 
\end{proof} 
\vskip 3mm
 The following result provides a  way to calculate the support function of $\mathsf{M}_{\delta}(\mathsf{epi}(\psi))$ when $\psi\in \mathfrak{C}$  is a  supercoercive convex function, that is $\psi\in \mathfrak{C}$ satisfies
 \begin{align} \label{super-co} \lim_{\|x\|\rightarrow\infty} \frac{\psi(x)}{\|x\|}=+\infty.\end{align}
 \vskip 2mm
 \begin{cor} \label{cor-super-co-co} Suppose that $\psi\in \mathfrak{C}$ is a  supercoercive convex function. 
 Then, for any $\theta\in \mathbb{S}^n$ such that $\langle \theta, e_{n+1}\rangle >0$ and any $\delta>0$, one has 
that  $y_{\theta}\in \partial \big(\mathsf{M}_{\delta}(\mathsf{epi}(\psi))\big)$ and  \eqref{supp-22-02} holds. 
 \end{cor}
 
 \begin{proof} 
 In view of Proposition \ref{pro1}, it is enough to prove $$\Sigma_{\psi}=\big\{ \theta\in \mathbb{S}^n: \ \langle \theta, e_{n+1}\rangle >0 \big\}. $$ To this end, let  $\theta=(\tilde{\theta}, \theta_{n+1})$ with $\theta_{n+1}>0$ and  $a\in \mathbb{R}$ be fixed. If $\mathsf{epi}(\psi) \cap \mathsf{H}^{-1}(\theta, a)\neq \emptyset$ and $(x, t)\in \mathsf{epi}(\psi) \cap \mathsf{H}^{-1}(\theta, a)$,  then $t\geq \psi(x)$ and 
   $$\langle (x,\psi(x)), (\tilde{\theta}, \theta_{n+1})\rangle =\langle x, \tilde{\theta}\rangle + \theta_{n+1}\psi(x) \leq \langle x, \tilde{\theta}\rangle + \theta_{n+1} t= \langle (x,t), (\tilde{\theta}, \theta_{n+1})\rangle \leq a.$$ Dividing by $\|x\|$ and letting $\|x\|\to \infty$, it follows from  \eqref{super-co} that $$ +\infty= \lim_{\|x\|\rightarrow \infty} \bigg(\Big\langle \frac{x}{\|x\|},  \tilde{\theta} \Big \rangle +\theta_{n+1} \cdot \frac{\psi(x)}{\|x\|}\bigg) \leq  \lim_{\|x\|\rightarrow \infty} \frac{a}{\|x\|}=0.$$ This is impossible. Hence, $\mathsf{epi}(\psi) \cap \mathsf{H}^{-1}(\theta, a)$, if not an empty set,  contains all $(x, t)$ with $x\in E$ for some bounded set $E\subset \mathbb{R}^n$.  This in turn implies that $$\psi(x) \leq t\leq \frac{a}{\theta_{n+1}}+\frac{\|\tilde{\theta}\|}{\theta_{n+1}} \cdot \sup_{x\in E} \|x\|<\infty. $$ Consequently, either $\mathsf{vol}_{n+1}(\mathsf{epi}(\psi)\cap\mathsf{H}^{-}(\theta,a))=0$ if   $\mathsf{epi}(\psi) \cap \mathsf{H}^{-1}(\theta, a)=\emptyset$, or $\mathsf{vol}_{n+1}(\mathsf{epi}(\psi)\cap\mathsf{H}^{-}(\theta,a))<\infty$ if the set $\mathsf{epi}(\psi) \cap \mathsf{H}^{-1}(\theta, a)$ is a  bounded  nonempty set. 
 \end{proof} 
\vskip 2mm
 The next result establishes the affine-invariance of the Ulam floating functions for convex and  log-concave functions. 
\vskip 2mm 
\begin{prp}\label{affine-inv-map}
		Let $\psi\in  \mathfrak{C}$ and $T:\mathbb{R}^n\rightarrow \mathbb{R}^n$ be an invertible linear map.  Assume that $f=e^{-\psi}$ is integrable. Define $(f\circ T)(x)=f(Tx)$. For any $\delta\ge 0$, one has 
		\[\mathsf{M}_{\delta}(\psi \circ T)=(\mathsf{M}_{\delta |\det T|}(\psi))\circ T  \ \ \ \mathrm{ and } \ \ \  U_{\delta}(f\circ T)=(U_{\delta |\det T|}(f))\circ T.
		\] In particular, if $|\det T|=1$, then \[\mathsf{M}_{\delta}(\psi \circ T)=(\mathsf{M}_{\delta}(\psi))\circ T  \ \ \ \mathrm{ and } \ \ \  U_{\delta}(f\circ T)=(U_{\delta}(f))\circ T.
		\]
	\end{prp}
\vskip 2mm
	\begin{proof} It is enough to only prove the statement \[\mathsf{M}_{\delta}(\psi \circ T)=(\mathsf{M}_{\delta |\det T|}(\psi))\circ T.
		\] Let $\widetilde{T}$ be the map defined by  $
		\widetilde{T}(x, x_{n+1})=(Tx, x_{n+1}).$ It can be checked that  $$\widetilde{T}^{-1}(x, x_{n+1})=(T^{-1}x,x_{n+1}).$$ Thus,   $
		\mathsf{epi}(\psi\circ T)=\widetilde{T}^{-1}(\mathsf{epi}(\psi)).
		$  According to Definition \ref{def-M-CC}, if $g\in \mathcal{F}_{\mathsf{epi}(\psi\circ T)},$ then $\widetilde{g}= g\circ (\widetilde{T}^{-1})$ satisfies that $\widetilde{g}: \mathbb{R}^{n+1}\to [0, 1]$,  $\int _{\mathsf{epi}(\psi)} \widetilde{g}(\hat{z})\,d\hat{z}=\delta\cdot |\det T|$, and   $$\int _{\mathsf{epi}(\psi\circ T)} zg(z)\,dz=\widetilde{T}^{-1} \Big(\int _{\mathsf{epi}(\psi)}\frac{ \hat{z}\widetilde{g}(\hat{z})}{|\det T|} \,d\hat{z}\Big).
		$$ 
		This further yields that  \[ \mathsf{M}_{\delta}
		(\mathsf{epi}(\psi\circ T))=\mathsf{M}_{\delta}(\widetilde{T}^{-1}(\mathsf{epi}
		(\psi)))=\widetilde{T}^{-1}(\mathsf{M}_{\delta   |\det T|}(\mathsf{epi}(\psi))).  \]
	It follows from   the definition of $\mathsf{M}_{\delta}(\psi)$  that 
	$$\mathsf{epi}(\mathsf{M}_{\delta}(\psi \circ T))
	=\mathsf{M}_{\delta}(\mathsf{epi}(\psi\circ T))  =\widetilde{T}^{-1}(\mathsf{M}_{\delta |\det T| }(\mathsf{epi}(\psi)))=
		\widetilde{T}^{-1}(\mathsf{epi}(\mathsf{M}_{\delta  |\det T|}(\psi))).
		$$  
	   Thus, for any $z=(x, s)\in \mathsf{epi}(\mathsf{M}_{\delta}(\psi \circ T))$, one obtains that 
	   $$s\geq (\mathsf{M}_{\delta}(\psi \circ T))(x).$$ The point $(x, s)$ is also in $\widetilde{T}^{-1}(\mathsf{epi}(\mathsf{M}_{\delta  |\det T|}(\psi)))$, which further implies  
	   $$
	   \widetilde{T}(x, s)=(Tx, s)\in \mathsf{epi}(\mathsf{M}_{\delta  |\det T|}(\psi)), 
	   $$ and thus $s\geq (\mathsf{M}_{\delta  |\det T|}(\psi))(Tx)$. 
	   This can hold if and only if 
	   $$
	   (\mathsf{M}_{\delta}(\psi \circ T))(x)=(\mathsf{M}_{\delta  |\det T|}(\psi))(Tx).
	   $$ 
	   As  $x$ is arbitrary, one immediately gets
	    \[
		\mathsf{M}_{\delta}(\psi \circ T)=(\mathsf{M}_{\delta |\det T|}(\psi))\circ T.
		\]  This completes the proof. 
	   		\end{proof}

 \vskip 3mm

\section{Ulam floating function and   affine surface area}\label{ul-4}
 
The section is devoted to the proof of Theorem \ref{mainthem2}.  We will  need more  preparation and several preliminary results. 
\par
We first note that,  under the assumptions of Theorem \ref{mainthem2},  the integral term in Theorem \ref{mainthem2}  is finite.  To see that  we recall the definition of the {\em rolling function} $r_C: \partial C \to [0, \infty)$ of a closed convex set $C$ in $\mathbb{R}^{n+1}$, which was 
introduced in \cite{Werner:1990}, see also \cite{Werner:2019},  as follows:
If   the outer unit normal  $N_{C}(z)$ at $z \in \partial C$ is unique, then 
$r_{C}(z)$ is the radius of the biggest Euclidean ball contained in $C$ that touches $C$ at $z$,
\[
r_{C}(z)=\max\{\rho:\mathbf{B}_2^{n+1}(z-\rho N_{C}(z),\rho)\subseteq C\}.
\] 
If $N_{C}(z)$ is not unique, then $r_{C}(z)=0$. 
In particular, if  $C=\mathsf{epi}(\psi)$ with $\psi\in \mathfrak{C}$, 
we write 
\begin{equation}\label{99}
	r_{\psi}(x)=r_{\mathsf{epi}(\psi)}(z_x),
\end{equation}  where   $z_x=(x,\psi(x)) \in \partial \mathsf{epi}(\psi)$. 
\vskip 2mm
Let $\psi\in \mathfrak{C}$. Since $\psi$ is continuous,   $\mathsf{epi}(\psi)$ is a closed set.
For functions $\psi$ such that $e^{-\psi}$ is integrable, we have
that $r_{\psi}(z)$ is bounded for every $z\in \partial \mathsf{epi}(\psi)$.
It follows that $\mathsf{epi}(\psi)$ contains a Euclidean ball
with radius $r_{\psi}(z)$ that has $z$ as an element. 
\par
Let now $\psi\in \mathfrak{C}$ be a convex function satisfying the above assumptions and
let  $z_{x}=(x,\psi(x))$. Recall that  $\kappa_{\psi}(z_x)$  and $N_{\psi}(z_x)$ are  the Gauss curvature and the outer unit normal of $\partial\mathsf{epi}(\psi)$ at $z_x$, respectively.  
We have for almost all $x \in \mathbb{R}^n$ (see e.g., \cite{Werner:2019}) that 
\begin{align}\label{curvature-cont-1} 
r_{\psi}(x)\leq (\kappa_{\psi}(z_x))^{-\frac{1}{n}}=
\frac{(1+\|\nabla\psi(x)\|^2)^\frac{n+2}{2n}}{(\det\nabla^2\psi(x))^\frac{1}{n}}.
\end{align}    
\par
\noindent
It has been proved in \cite[Lemma 8]{Werner:2019} that, if  $\psi:\mathbb{R}^n\to\mathbb{R}$ is a convex function such that 
	$e^{-\psi}$ is integrable, then for all $0\leq\alpha<1$,
	\begin{align}\label{integrable-001} 
	\int_{\mathbb{R}^n}\frac{(1+\|\nabla\psi(x)\|^2)^\frac{1}{2}}{r_{\psi}(x)^\alpha}e^{-\psi(x)}dx<\infty.
	\end{align} 
	In particular, \eqref{integrable-001} holds for $\alpha=\frac{n}{n+2}$.  
Together with \eqref{curvature-cont-1}, one gets  
\begin{align}\label{integrable-002} 
\int_{\mathbb{R}^n}(\det(\nabla^2(\psi(x)))^\frac{1}{n+2}e^{-\psi(x)}
dx\leq\int_{\mathbb{R}^n}\frac{(1+\|\nabla\psi(x)\|^2)^\frac{1}{2}}{r_{\psi}(x)^{\frac{n}{n+2}}}e^{-\psi(x)}\,dx<\infty.
\end{align}  
If the Gauss curvature $\kappa_{\psi}(z_x)$  is $0$ almost everywhere, then all expressions in the  identities of Theorem  \ref{mainthem2}  are $0$.
This is in particular the case when $\psi$ is piecewise affine.   
\vskip 3mm
In  the following lemma we write in short  $\rho\mathbf{B}_2^{n+1}= \mathbf{B}_2^{n+1}(o, \rho)$. 
Note that, for $\delta>0$ small enough, $\mathsf{M}_{\delta}(\rho\mathbf{B}_2^{n+1})$ is again a Euclidean ball  centered at the origin $o$.
\vskip 2mm
\begin{lem}\label{11}  
Let  $\delta>0$  be small enough and  $\Delta_h$ be the height of a cap of 
$\rho\mathbf{B}_2^{n+1}$ of volume   $\delta$. Let $\Delta_\rho$ be the difference of the radii of 
$\rho\mathbf{B}_2^{n+1}$ and $\mathsf{M}_{\delta}(\rho\mathbf{B}_2^{n+1})$.   Then, 
	\[
	\lim\limits_{\delta\to0^+}\frac{\Delta_h}{\Delta_{\rho}}=
	\frac{n+4}{n+2}.
	\]
\end{lem}
\begin{proof} Let $\Delta_h$ be the height of a cap of 
$\rho\mathbf{B}_2^{n+1}$ of volume  $\delta$. With the   change  of variable $y=(z, \rho-t)\in \rho\mathbf{B}_2^{n+1}$ with $t\in (0, \rho)$ and $z\in (\rho^2-(\rho-t)^2)^\frac{1}{2}\mathbf{B}_2^{n}=(2\rho t-t^2)^{\frac{1}{2}}\mathbf{B}_2^{n}$,  the following hold:   
\begin{align*} &\delta  =\mathsf{vol}_{n+1}(\rho\mathbf{B}_2^{n+1}\cap\mathsf{H}^{+}(e_{n+1},\Delta_h))  
		=\mathsf{vol}_n(\mathbf{B}_2^n) \int_{0}^{\Delta_h}
		(2\rho t-t^2)^\frac{n}{2}\,dt,  \\ &\int_{\rho\mathbf{B}_2^{n+1}\cap\mathsf{H}^{+}(e_{n+1},\Delta_h)} y\,dy =e_{n+1}\cdot  \mathsf{vol}_n(\mathbf{B}_2^n) \int_{0}^{\Delta_h}
		(2\rho t-t^2)^\frac{n}{2}(\rho-t)\,dt.
\end{align*} 
These, together with \eqref{f-17} and $d\nu_{\mathsf{epi}(\psi),  \delta} =\delta^{-1}\mathbf{1}_{\mathsf{epi}(\psi)}dx,$ 
 result in 
	\[
	\Delta_{\rho}=\frac{\int_{0}^{\Delta_h}(2\rho t-t^2)^\frac{n}{2}t\,dt}{\int_{0}^{\Delta_h}(2\rho t-t^2)^\frac{n}{2}\,dt}.
	\] Clearly, if $\delta\rightarrow 0^+$, then $\Delta_h\to0^+$. It follows from  L'Hospital's rule that 
		\begin{align*}
			\lim_{\delta\to0^+}\frac{\Delta_h}{\Delta_\rho}&=\lim_{\Delta_h\to0^+}
			\frac{\Delta_h \int_{0}^{\Delta_h} (2\rho t-t^2)^\frac{n}{2}\,dt}{\int_{0}^{\Delta_h}(2\rho t-t^2)^\frac{n}{2}t\,dt} =1+\lim_{\Delta_h\to0^+}\frac{\ \int_{0}^{\Delta_h} (2\rho t-t^2)^\frac{n}{2}\,dt}{(2\rho\Delta_h-\Delta_h^2)^\frac{n}{2}\Delta_h}\\&=1+\lim_{\Delta_h\to0^+}\frac{(2\rho\Delta_h -\Delta_h^2)^\frac{n}{2}}
			{(2\rho\Delta_h -\Delta_h^2)^\frac{n}{2}+n(2\rho\Delta_h -\Delta_h^2)^\frac{n-2}{2} (\rho-\Delta_h)\Delta_h}			\\&=1+\lim_{\Delta_h\to0^+}\frac{(2\rho -\Delta_h )}
			{(2\rho  -\Delta_h )+n(\rho-\Delta_h)}	 =\frac{n+4}{n+2}.
		\end{align*}  This completes the proof.
\end{proof}
\vskip 3mm
The following lemma will be needed in order to use the Dominated Convergence Theorem  in the proof of  Theorem \ref{mainthem2}. 
\vskip 2mm
\begin{lem}\label{mainlem1}
	Let $\psi\in \mathfrak{C}$ be such that 
	$0< \int_{\mathbb{R}^n}e^{-\psi(x)}dx<\infty$. There
	exist $\delta_0>0$ and a constant  $\alpha_n>0$ such that 	\[
	0\leq\frac{\mathsf{M}_{\delta}(\psi)(x)-\psi(x)}{\delta^\frac{2}{n+2}}
	\leq  \alpha_n \frac{(1+\|\nabla\psi(x)\|^2)^\frac{1}{2}}{r_{\psi}(x)^{\frac{n}{n+2}}},
	\] holds for all $\delta\in (0, \delta_0)$ and  for all $x\in\mathbb{R}^n$, 
	where $r_{\psi}(x)$ is given by \eqref{99}. \end{lem} 
\vskip 2mm		
	\begin{proof}   Let $\psi_{\delta}$ be as in \eqref{def-floating-f-1}.  It has been proved in   \cite[Lemma 5]{Werner:2019}  that there
	exist $\delta_0>0$ and a constant  $\alpha_n>0$  such that 	\[
	0\leq\frac{\psi_{\delta}(x)-\psi(x)}{\delta^\frac{2}{n+2}}
	\leq \alpha_n \frac{(1+\|\nabla\psi(x)\|^2)^\frac{1}{2}}{r_{\psi}(x)^{\frac{n}{n+2}}},
	\] holds for all $\delta\in\ (0, \delta_0)$ and  for all $x\in\mathbb{R}^n$.  It can be checked that for any $\delta\in (0, \delta_0)$, $(\mathsf{epi}(\psi))_{\delta}\subseteq\mathsf{M}_{\delta}
		(\mathsf{epi}(\psi))$, where $(\mathsf{epi}(\psi))_{\delta}$ is the floating set of $\mathsf{epi}(\psi)$ defined in \eqref{floating-epi}. 	Hence, $\psi_{\delta}\geq \mathsf{M}_{\delta}(\psi)$, and this in turn further yields \[
	0\leq\frac{\mathsf{M}_{\delta}(\psi)(x)-\psi(x)}{\delta^\frac{2}{n+2}}
	\leq  \alpha_n \frac{(1+\|\nabla\psi(x)\|^2)^\frac{1}{2}}{r_{\psi}(x)^{\frac{n}{n+2}}},
	\]  for all $\delta\in\ (0, \delta_0)$ and  for all $x\in\mathbb{R}^n$.  
	\end{proof} 
\vskip 3mm	
We apply 	Lemma \ref{mainlem1} to the log-concave function $f=e^{-\psi}$ where  $\psi\in \mathfrak{C}$ is such that 
	$0< \int_{\mathbb{R}^n}e^{-\psi(x)}dx<\infty$. As $1-e^{-t}\leq t$ for all $t\in (0, \infty)$, one gets 
	$$ f-U_{\delta}(f)=e^{-\psi}-e^{-\mathsf{M}_{\delta}(\psi)} =e^{-\psi}\cdot \big(1-e^{-(\mathsf{M}_{\delta}(\psi)-\psi)} \big)\leq e^{-\psi}\cdot \big(\mathsf{M}_{\delta}(\psi)-\psi \big).$$ 
	It then follows from Lemma \ref{mainlem1}  that there
	exist $\delta_0>0$ and a constant  $\alpha_n>0$ such that  
	\[
	0\leq \frac{f-U_{\delta}(f)}{\delta^\frac{2}{n+2}}\leq \frac{e^{-\psi}\cdot \big(\mathsf{M}_{\delta}(\psi)-\psi \big)}{\delta^\frac{2}{n+2}} \leq \alpha_n \frac{(1+\|\nabla\psi(x)\|^2)^\frac{1}{2}}{r_{\psi}(x)^{\frac{n}{n+2}}}f(x),
	\] holds for all $\delta\in\ (0, \delta_0)$ and  for all $x\in\mathbb{R}^n$. Together with \eqref{integrable-001}  and the Dominated Convergence Theorem, one gets \begin{align} \lim_{\delta\to0^+} \bigg(\delta^{-\frac{2}{n+2}}  \int_{\mathbb{R}^n}\big( f-   U_{\delta}(f)\big)\,dx \bigg)& = \lim_{\delta\to0^+}\bigg(\delta^{-\frac{2}{n+2}} \int_{\mathbb{R}^n}(e^{-\psi(x)}-e^{-\mathsf{M}_{\delta}(\psi)(x)})dx\bigg) \nonumber \\ &= \int_{\mathbb{R}^n} \lim_{\delta\to0^+}\frac{ (e^{-\psi(x)}-e^{-\mathsf{M}_{\delta}(\psi)(x)})}{\delta^{\frac{2}{n+2}}}\,dx. \label{dct-11}  \end{align} 
	Similarly, the following formulas hold: 
	\begin{align} \lim_{\delta\to0^+} \bigg(\delta^{-\frac{2}{n+2}}  \int_{\mathbb{R}^n}\big|
	\mathsf{M}_{\delta}(\psi)(x)-\psi(x)\big|e^{-\psi(x)}\,dx\bigg)& = \int_{\mathbb{R}^n}  \lim_{\delta\to0^+} \frac{|
	\mathsf{M}_{\delta}(\psi)(x)-\psi(x)|e^{-\psi(x)}}{\delta^{\frac{2}{n+2}}} \,dx \nonumber\\ &= \int_{\mathbb{R}^n} \lim_{\delta\to0^+}\frac{ (e^{-\psi(x)}-e^{-\mathsf{M}_{\delta}(\psi)(x)})}{\delta^{\frac{2}{n+2}}}\,dx. \label{dct-11-22}  \end{align}

\vskip 3mm 	

Let us recall the  main  theorem.  
\vskip 2mm\noindent {\bf Theorem \ref{mainthem2}.}  {\em
	Let $\psi:\mathbb{R}^n\to\mathbb{R}$ be a convex function such that
	\[
	0<\int_{\mathbb{R}^n}e^{-\psi(x)}dx<\infty.
	\]
 Then  
\begin{align} 
\lim_{\delta\to0^+} \bigg(\delta^{-\frac{2}{n+2}} \int_{\mathbb{R}^n}(e^{-\psi(x)}-e^{-\mathsf{M}_{\delta}(\psi)(x)})\,dx\bigg)&=  \lim_{\delta\to0^+} \bigg(\delta^{-\frac{2}{n+2}}  \int_{\mathbb{R}^n}\big|
	\mathsf{M}_{\delta}(\psi)(x)-\psi(x)\big|e^{-\psi(x)}\,dx\bigg) \nonumber \\&=c_{n+1}\int_{\mathbb{R}^n}\big(\det(\nabla^2\psi(x))\big)^{\frac{1}{n+2}}e^{-\psi(x)}dx,\label{main-theo--22}
\end{align} where $c_{n+1}$ is the constant given in \eqref{003}, i.e.,   
\begin{equation*} 
	c_{n+1}=\frac{n+2}{2(n+4)}\Big(\frac{n+2}{\mathsf{vol}_n(\mathbf{B}_2^n)}\Big)^{\frac{2}{n+2}}.
\end{equation*}}
\vskip 2mm
\begin{proof}  	
Let $z_x=(x,\psi(x))$ as above.	According to \eqref{dct-11} and
 \eqref{dct-11-22}, the desired formula  \eqref{main-theo--22} will hold  if the following limit is verified.  For almost all $x\in \mathbb{R}^n$,  \begin{align} \lim_{\delta\to 0^+} \frac{\mathsf{M}_{\delta}(\psi)(x)-\psi(x)}{\delta^{\frac{2}{n+2}}} =c_{n+1}\cdot \big(\det(\nabla^2\psi(x))\big)^{\frac{1}{n+2}}. \label{desied-lim-1} \end{align} As $\nabla^2\psi$ exists almost everywhere, the proof of \eqref{desied-lim-1}  can be separated into two cases. 	
	\vskip 2mm 
	\noindent {\it Case 1: The Hessian matrix  $\nabla^2\psi (x)$ is positive definite and  $N_{\psi}(z_x)$
		exists uniquely at $x\in\mathbb{R}^n$.} 
		 
\par
\noindent		
		 In this case, \eqref{desied-lim-1}  follows immediately from the following claim:  
	for all $\varepsilon>0$ small enough, there exists $\delta_1=\delta_1(x,\varepsilon)$ such that for all 
	$0<\delta<\delta_1$,
	\begin{align}
	(1-\varepsilon)^3
	\leq\frac{\mathsf{M}_{\delta}(\psi)(x)-\psi(x)}{c_{n+1}\cdot  \delta^{\frac{2}{n+2}}\cdot  \det(\nabla^2\psi(x))^{\frac{1}{n+2}} } 
	\leq (1-\varepsilon)^{-5}. \label{control-1}
	\end{align} 
\par
\noindent	
	We first assume that  the indicatrix of Dupin at $z_x\in\partial\mathsf{epi}(\psi)$ is a sphere of radius $\sqrt{\rho}$. Then  (see, e.g.,   \cite[ equation (5)]{Schutt:2003}) 
	   \begin{equation}\label{eqqq1}
	\rho=(\kappa_{\psi}(z_x)) ^{-\frac{1}{n}}.
	\end{equation}  
	It follows from \eqref{qu1} that  
	\begin{equation}\label{w0} \left<-N_{\psi}(z_{x}),e_{n+1}\right>=(1+\|\nabla\psi(x)\|^2)^{-\frac{1}{2}},~~
	\rho^{-\frac{n}{n+2}}(1+\|\nabla\psi(x)\|^2)^{\frac{1}{2}}=\det(\nabla^2\psi(x))^{\frac{1}{n+2}}. 
	\end{equation} 
And,  see \cite{Schutt:2003}, locally around $z_x$,   $\mathsf{epi}(\psi)$ can be approximated 
by  $\mathbf{B}_2^{n+1}(z_x-\rho N_{\psi}(z_x),\rho)$. 
\par	
	We now describe the approximation in more details. Let $ \varepsilon >0$ be a given small enough number.
	For simplicity, let  \begin{align*}
	\mathcal{B}&=\mathbf{B}_2^{n+1}(z_x-\rho N_{\psi}(z_x),\rho),\\ 
	\mathcal{B}(\varepsilon^{-})&=\mathbf{B}_2^{n+1}(z_x-(1-\varepsilon)\rho N_{\psi}(z_x),(1-\varepsilon)\rho),\\
	\mathcal{B}(\varepsilon^{+})&=\mathbf{B}_2^{n+1}(z_x-(1+\varepsilon)\rho N_{\psi}(z_x),(1+\varepsilon)\rho).
	\end{align*}    
It  is easily checked that  there is a  $\Xi_{\varepsilon}>0$  such that, (see e.g., the proof of  \cite[Lemma 23]{SchuettWerner:2004}), for all $0<t\leq\Xi_{\varepsilon}$, \begin{align} \mathsf{H}^{-}(-N_{\psi}(z_x),t)\cap\mathcal{B}(\varepsilon^{-})  \subseteq\mathsf{H}^{-}(-N_{\psi}(z_x),t)\cap\mathsf{epi}(\psi) 
	\subseteq\mathsf{H}^{-}(-N_{\psi}(z_x),t)\cap\mathcal{B}(\varepsilon^{+}). \label{equ0} 
	\end{align} This further implies  
	\begin{equation} \label{key}
	\begin{aligned}
	\mathsf{M}_{\delta}\big(\mathsf{H}^{-}(-N_{\psi}(z_x),t)\cap\mathcal{B}(\varepsilon^{-})\big) &\subseteq\mathsf{M}_{\delta}\big(\mathsf{H}^{-}(-N_{\psi}(z_x),t)\cap\mathsf{epi}(\psi)\big)\\&
	\subseteq\mathsf{M}_{\delta}\big(\mathsf{H}^{-}(-N_{\psi}(z_x),t)\cap\mathcal{B}(\varepsilon^{+})\big).
	\end{aligned}
	\end{equation}
\par	
	Although Lemma \ref{lem: UnaffectedCut} was proved for convex bodies, it can be checked that the same  result also holds for $\mathsf{epi}(\psi)$ with $z_x\in \partial \mathsf{epi}(\psi)$. Indeed, Lemma \ref{lem: UnaffectedCut} is local around the point $z_x$. 
	Thus we apply Lemma \ref{lem: UnaffectedCut} with $t=\Xi_{\varepsilon}$ to $\mathcal{B}(\varepsilon^{-}), \mathsf{epi}(\psi)$ and $\mathcal{B}(\varepsilon^{+})$, respectively. Hence  there is $r>0$ such that for all $\delta>0$,
	\begin{equation}\label{le0}
	\mathsf{M}_{\delta}\big(\mathcal{B}(\varepsilon^{-})\cap\mathsf{H}^{-}(-N_{\psi}(z_x),\Xi_{\varepsilon})\big)\cap\mathbf{B}_2^{n+1}(z_x,r)
	=\mathsf{M}_{\delta}(\mathcal{B}(\varepsilon^{-}))\cap\mathbf{B}_2^{n+1}(z_x,r),
	\end{equation}
	\begin{equation}\label{le1}
	\mathsf{M}_{\delta}\big(\mathsf{epi}(\psi) \cap\mathsf{H}^{-}(-N_{\psi}(z_x),\Xi_{\varepsilon})\big)\cap\mathbf{B}_2^{n+1}(z_x,r)
	=\mathsf{M}_{\delta}(\mathsf{epi}(\psi))\cap\mathbf{B}_2^{n+1}(z_x,r),
	\end{equation}
	and 
	\begin{equation}\label{le2}
	\mathsf{M}_{\delta}\big(\mathcal{B}(\varepsilon^{+})\cap\mathsf{H}^{-}(-N_{\psi}(z_x),\Xi_{\varepsilon})\big)\cap\mathbf{B}_2^{n+1}(z_x,r)
	=\mathsf{M}_{\delta}(\mathcal{B}(\varepsilon^{+}))\cap\mathbf{B}_2^{n+1}(z_x,r).
	\end{equation}
	
	We choose $\Theta_{\delta}<\Xi_{\varepsilon}$ small enough that
	\begin{equation}\label{1}
	\mathsf{H} ^{-}(-N_{\psi}(z_x), \Theta_{\delta})\subseteq \mathsf{H}^{-}(-N_{\psi}(z_x),\Xi_{\varepsilon})\ \ \mathrm{and}\ \ \mathsf{H} ^{-}(-N_{\psi}(z_x), \Theta_{\delta}) \cap\mathcal
	{B}(\varepsilon^{+})\subseteq\mathbf{B}_2^{n+1}(z_x,r).
	\end{equation} 
	We intersect  both sides of \eqref{le0}, \eqref{le1} and \eqref{le2} with sets $\mathsf{H} ^{-}(-N_{\psi}(z_x), \Theta_{\delta})\cap\mathcal
	{B}(\varepsilon^{-})$, $\mathsf{H} ^{-}(-N_{\psi}(z_x), \Theta_{\delta})\cap\mathsf{epi}(\psi)$ and $\mathsf{H} ^{-}(-N_{\psi}(z_x), \Theta_{\delta})\cap\mathcal
	{B}(\varepsilon^{+})$, respectively. Thus
	\[
	\mathsf{M}_{\delta}\big(\mathcal{B}(\varepsilon^{-})\cap\mathsf{H}^{-}(-N_{\psi}(z_x),\Xi_{\varepsilon})\big)\cap\mathsf{H} ^{-}(-N_{\psi}(z_x), \Theta_{\delta})
	=\mathsf{M}_{\delta}(\mathcal{B}(\varepsilon^{-}))\cap\mathsf{H} ^{-}(-N_{\psi}(z_x), \Theta_{\delta}),
	\]
	\[
	\mathsf{M}_{\delta}\big(\mathsf{epi}(\psi) \cap\mathsf{H}^{-}(-N_{\psi}(z_x),\Xi_{\varepsilon})\big)\cap\mathsf{H} ^{-}(-N_{\psi}(z_x), \Theta_{\delta})
	=\mathsf{M}_{\delta}(\mathsf{epi}(\psi))\cap\mathsf{H} ^{-}(-N_{\psi}(z_x), \Theta_{\delta})
	\]
	and
	\[
	\mathsf{M}_{\delta}\big(\mathcal{B}(\varepsilon^{+})\cap\mathsf{H}^{-}(-N_{\psi}(z_x),\Xi_{\varepsilon})\big)\cap\mathsf{H} ^{-}(-N_{\psi}(z_x), \Theta_{\delta})
	=\mathsf{M}_{\delta}(\mathcal{B}(\varepsilon^{+}))\cap\mathsf{H} ^{-}(-N_{\psi}(z_x), \Theta_{\delta}).
	\]
	By  \eqref{key} we get
	\[\begin{aligned}
	\mathsf{M}_{\delta}(\mathcal{B}(\varepsilon^{-}))\cap \mathsf{H} ^{-}(-N_{\psi}(z_x), \Theta_{\delta}) &\subseteq\mathsf{M}_{\delta}
	(\mathsf{epi}(\psi))\cap\mathsf{H} ^{-}(-N_{\psi}(z_x), \Theta_{\delta}) \\&
	\subseteq\mathsf{M}_{\delta}(\mathcal{B}(\varepsilon^{+}))\cap
	\mathsf{H} ^{-}(-N_{\psi}(z_x), \Theta_{\delta}).
	\end{aligned}\] 
	Let $z_{\delta}=(x,\mathsf{M}_{\delta}(\psi)(x))$. Choose $\delta$ so small that $z_{\delta}\in\mathsf{H} ^{-}(-N_{\psi}(z_x), \Theta_{\delta})$ and $z_{\delta}\in\mathcal{B}(\varepsilon^{-})$.
	Clearly, in Figure \ref{figure-2} we see that 
	\begin{align} z_{\delta} \in\mathsf{int}\big(\mathsf{M}_{\delta}(\mathcal{B}(\varepsilon^{+}))\cap
	\mathsf{H} ^{-}(-N_{\psi}(z_x), \Theta_{\delta})\big) \ \ \mathrm{and} \ \  z_{\delta} \notin\mathsf{int}\big(\mathsf{M}_{\delta}(\mathcal{B}(\varepsilon^{-}))\cap
	\mathsf{H} ^{-}(-N_{\psi}(z_x), \Theta_{\delta})\big). \label{location-B}
	\end{align}

\begin{figure}[htpb] 
	\centering
	\begin{tikzpicture}[node distance=2cm]
	\draw[->] (-8,0) -- (8,0) node[below] {$\mathbb{R}^{n}$};
	\draw[->] (0,-0.5) -- (0,9.5) node[left] {$\mathbb{R}$};
	\draw[-] (-0.3,0) -- (-0.3,0)node[below]{$o$};
	\draw[domain=-8.5:8.5] plot(\x, 1/14*\x*\x+1/14)node[right]{$\mathsf{epi}(\psi)$};
	\filldraw[fill=black] (-1,0) circle (0.06) node[below] {$x$};
	\filldraw[fill=black] (-1,1/7) circle (0.06) node[right] {$z_x$};
	\draw[dashed,-] (-1,0)--(-1,1/7);
	\draw[dashed,-] (-2,2/7)-- (0,0);
	\draw[dashed,->] (-1,1/7)-- (1/3,199/21) node[right] {$-N_{\psi}(z_x)$};
	\draw[-] (-5.96,5.86) arc (-180:0:50^0.5/49*40) node[below] {$\mathcal{B}(\varepsilon^{-})$};
	\draw[-] (-8,6.01) arc (-161.8:-27:50^0.5/49*60) node[below] {$\mathcal{B}(\varepsilon^{+})$};
	\draw[dashed,-] (-8,37/14)-- (7,0.5) node[above] {$\mathsf{H}^{-}(-N_{\psi}(z_x),\Theta_{\delta})$};
	\filldraw[fill=black] (-1,4/7) circle (0.06) node[right] {$z_{\delta}$};
	\draw[-] (-1,0)--(-1,3.5/7);
	\draw[dashed,-] (-5.38,5.86) arc (-180:0:50^0.5/49*36) node[above] {$\mathsf{M}_{\delta}(\mathcal{B}(\varepsilon^{-}))$};
	\draw[dashed,-] (-7.7258,6.1) arc (-161.8:-27:50^0.5/49*58) node[above] {$\mathsf{M}_{\delta}(\mathcal{B}(\varepsilon^{+}))$};
	\filldraw[fill=black] (-9/49,41/7) circle (0.06) node[right] {$z_x-(1-\varepsilon)\rho N_{\psi}(z_x)$};
	\filldraw[fill=black] (11/49,61/7) circle (0.06) node[right] {$z_x-(1+\varepsilon)\rho N_{\psi}(z_x)$};
	\end{tikzpicture} 
	\caption{}  \label{figure-2}
\end{figure} 
\vskip 2mm	
	
	Denote by $\Delta_{\rho^{+}}$ the difference of the radii of $\mathcal{B}(\varepsilon^{+})$ and $\mathsf{M}_{\delta}(\mathcal{B}(\varepsilon^{+}))$. Let 
	$\Delta_{h^{+}}$ be the height of a  cap of 
	$\mathcal{B}(\varepsilon^{+})$ which has volume $\delta$.
	It follows from Lemma \ref{11} that \[
	\lim_{\delta\to0^+}\frac{\Delta_{h^{+}}}{\Delta_{\rho^{+}}}=
	\frac{n+4}{n+2}.
	\]
	Hence, for $ \varepsilon >0$ small enough, there is  a constant  $\delta_2=\delta_2(x, \varepsilon)$ such that for all  $0<\delta\leq\delta_2$,	
	\begin{equation}\label{equat}
	\frac{n+4}{n+2}-\varepsilon\leq\frac{\Delta_{h^{+}}}{\Delta_{\rho^{+}}}\leq \frac{n+4}{n+2}+\varepsilon. 
	\end{equation}   Lemma \ref{lemma1} yields that 
	\begin{align}\label{control-delta-1} 
	\delta\leq  \big(d_{n+1}\big)^{-\frac{n+2}{2}} ((1+\varepsilon)\rho)^{\frac{n}{2}}
	(\Delta_{h^{+}})^{\frac{n+2}{2}}.	\end{align} 
	Moreover, we assume that for all $0<\delta\leq\delta_2$,
	\[
	\mathsf{dist}\big(z_{\delta}, \big(\mathcal{B}(\varepsilon^{+})\big)^c\big)-\mathsf{dist}\big(z_{\delta}, \mathcal{B}^c\big)\leq\varepsilon \, \mathsf{dist}\big(z_{\delta}, \mathcal{B}^c\big).
	\]
	From \eqref{location-B}, one gets \begin{align}\Delta_{ \rho^+} \leq \mathsf{dist}\big(z_{\delta}, \big(\mathcal{B}(\varepsilon^{+})\big)^c\big) \leq (1+\varepsilon) \mathsf{dist}\big(z_{\delta}, \mathcal{B}^c\big) . \label{control-rho+}   \end{align} 
	This, together with \eqref{equat},  \eqref{control-delta-1} and \eqref{control-rho+}, imply that 
	\begin{align}\label{control-delta-2} 
	d_{n+1} \delta^{\frac{2}{n+2}}  \leq    
	\big((1+\varepsilon)\rho\big)^{\frac{n}{n+2}}   \Big(\frac{n+4}{n+2}+\varepsilon\Big)\big( (1+\varepsilon) \mathsf{dist}(z_{\delta},  \mathcal{B}^c)\big).
	\end{align} 
	Applying Lemma \ref{lemma2}  with $o$, $z$, $-e_{n+1}$ and $z_n$ replaced by 
	$z_x$, $z_{\delta}$, $N_{\psi}(z_x)$ and
	$\big(\mathsf{M}_{\delta}(\psi)(x)-\psi(x)\big)\left<-N_{\psi}(z_x),e_{n+1}\right>$, respectively, 
	one gets that for $\delta>0$ small enough,  	 
	\begin{equation}
	\big(\mathsf{M}_{\delta}(\psi)(x)-\psi(x)\big)\left<-N_{\psi}(z_x),e_{n+1}\right>\ge \mathsf{dist}(z_{\delta},  \mathcal{B}^c).
	\end{equation}    
	Together with   \eqref{w0} and \eqref{control-delta-2}, we obtain the  following inequality 
	\begin{align*}  \frac{\mathsf{M}_{\delta}(\psi)(x)-\psi(x)}{d_{n+1} \delta^\frac{2}{n+2}}&\geq  \big((1+\varepsilon)\rho\big)
	^{-\frac{n}{n+2}} \Big[ \Big(\frac{n+4}{n+2}+\varepsilon\Big) \cdot  \frac{\big( (1+\varepsilon)\mathsf{dist}(z_{\delta},  \mathcal{B}^c)\big)}  { \mathsf{dist}(z_{\delta},  \mathcal{B}^c)(1+\|\nabla\psi(x)\|^2)^{\frac{1}{2}}}\Big]^{-1}   \nonumber \\&
	\geq \big((1+\varepsilon)\rho\big)
	^{-\frac{n}{n+2}} \Big(\frac{n+4}{n+2} \Big)^{-1} \bigg(1-  \Big( \frac{n+2} {n+4}\Big) \varepsilon\bigg)\frac{(1+\|\nabla\psi(x)\|^2)^{\frac{1}{2}}}{1+\varepsilon}\\&
	\ge(1+\varepsilon)^{-\frac{2n+2}{n+2}} \Big(\frac{n+4}{n+2} \Big)^{-1} \bigg(1-  \Big( \frac{n+2} {n+4}\Big) \varepsilon\bigg)\det(\nabla^2\psi(x))^{\frac{1}{n+2}}. 
	\end{align*} 
After rearrangement, and using  the relation between $c_{n+1}$ and $d_{n+1}$ (see \eqref{003} and \eqref{003-d}, respectively), one gets for all  $0<\delta\leq\delta_2$,
	\begin{align}
	\frac{\mathsf{M}_{\delta}(\psi)(x)-\psi(x)}{c_{n+1} \delta^\frac{2}{n+2}}  
	\geq  (1-\varepsilon)^3    \big(\det(\nabla^2\psi(x))\big) ^{\frac{1}{n+2}}, \label{upper-difference} 
	\end{align}
 where we have used $\frac{1}{1+a}\geq 1-a$ for all $a\in (0, 1)$.
\par 		
Let us now prove the upper bound in \eqref{control-1}.  Denote by 
$\Delta_{{\rho}^{-}}$ the difference of the radii of $\mathcal{B}(\varepsilon^{-})$ and
$\mathsf{M}_{\delta}(\mathcal{B}(\varepsilon^{-}))$. Let $\Delta_{h^{-}}$
be the height of a  cap of $\mathcal{B}(\varepsilon^{-})$ which has volume $\delta$.
It follows from Lemma \ref{11} that \[
	\lim_{\delta\to0^+}\frac{\Delta_{h^{-}}}{\Delta_{\rho^{-}}}=
	\frac{n+4}{n+2}.
	\]
 Hence, for $ \varepsilon>0$ small enough, there is  a constant  $\delta_3=\delta_3(x, \varepsilon)$ such that for all  $0<\delta\leq\delta_3$,	
	\begin{equation}\label{equqat1}
		\frac{n+4}{n+2}-\varepsilon\leq\frac{\Delta_{h^{-}}}{\Delta_{\rho^{-}}}\leq \frac{n+4}{n+2}+\varepsilon. 
	\end{equation}  
	Lemma \ref{lemma1} yields that 
	\begin{equation}
	\delta\ge \big(d_{n+1}\big)^{-\frac{n+2}{2}} \rho^{\frac{n}{2}}(1-\varepsilon)^{\frac{n}{2}}
	\Big(1-\frac{\Delta_{h^{-}}}{2(1-\varepsilon)\rho}\Big)^\frac{n}{2}
	(\Delta_{h^{-}})^\frac{n+2}{2}. \label{vol-12-31}
	\end{equation} 
		Moreover, we assume that for all $0<\delta\leq\delta_3$,
	\[
	 \mathsf{dist}\big(z_{\delta}, \mathcal{B}^c\big)-\mathsf{dist}\big(z_{\delta},\big(\mathcal{B}(\varepsilon^{-})\big)^c\big)\leq\varepsilon\, \mathsf{dist}\big(z_{\delta}, \mathcal{B}^c\big)
	\]
	and $\Delta_{h^{-}}<\rho \varepsilon$.
	From \eqref{location-B}, one gets \begin{align}\Delta_{ \rho^-} \ge \mathsf{dist}\big(z_{\delta}, \big(\mathcal{B}(\varepsilon^{-})\big)^c\big) \ge (1-\varepsilon) \mathsf{dist}\big(z_{\delta}, \mathcal{B}^c\big) . \label{control-rho-}   \end{align}
	This, together with \eqref{equqat1},   \eqref{vol-12-31}, \eqref{control-rho-}, $\Delta_{h^{-}}<\rho \varepsilon$ and  $ \varepsilon \in (0, (n+1)^{-2})$, imply that 
	\begin{align}
		  d_{n+1}   \delta^{\frac{2}{n+2}}& 
		  \geq \rho^{\frac{n}{n+2}}(1-\varepsilon)^{\frac{n}{n+2}}
		  \Big(1-\frac{\Delta_{h^{-}}}{2(1-\varepsilon)\rho}\Big)^\frac{n}{n+2}\Big(\frac{n+4}{n+2}-\varepsilon\Big)\big( (1-\varepsilon) \mathsf{dist}(z_{\delta},  \mathcal{B}^c)\big) \nonumber\\&
		  \geq \rho^{\frac{n}{n+2}}(1-\varepsilon)^{\frac{2n+2}{n+2}}
		\Big(1-\frac{\Delta_{h^{-}}}{\rho}\Big)^\frac{n}{n+2}
		\Big(\frac{n+4}{n+2}\Big) \Big(1-\frac{(n+2)\varepsilon}{n+4}\Big)\mathsf{dist}(z_{\delta},  \mathcal{B}^c) \nonumber\\ &
		\geq \rho^{\frac{n}{n+2}}(1-\varepsilon)^{\frac{4n+4}{n+2}}
		\Big(\frac{n+4}{n+2}\Big)\mathsf{dist}(z_{\delta},  \mathcal{B}^c).\label{control-delta-3}
	\end{align} 
		Applying  again Lemma \ref{lemma2}, 
	one gets that for $\delta>0$ small enough,  	 
	\begin{align}
	\big(\mathsf{M}_{\delta}(\psi)(x)-\psi(x)\big)\left<-N_{\psi}(z_x),e_{n+1}\right> &\leq \mathsf{dist}(z_{\delta}, 
\mathcal{B}^c)\Big(1+\frac{2\mathsf{dist}(z_{\delta}, 
	\mathcal{B}^c)}{\rho\left<-N_{\psi}(z_x),e_{n+1}\right>^2}\Big) \nonumber \\ &\leq
\mathsf{dist}(z_{\delta}, \mathcal{B}^c)(1+\varepsilon) \label{re1}.
	\end{align}   
	Together with \eqref{w0}, \eqref{control-delta-3} and the relation between $c_{n+1}$ and $d_{n+1}$ (see \eqref{003} and \eqref{003-d}, respectively), one has for all  $0<\delta\leq\delta_3$, 
		\begin{align}
			\frac{\mathsf{M}_{\delta}(\psi)(x)-\psi(x)}{c_{n+1} \delta^\frac{2}{n+2}}&
			\leq \rho^{-\frac{n}{n+2}}
			(1-\varepsilon)^{-\frac{4n+4}{n+2}}(1+\varepsilon)(1+\|\nabla\psi(x)\|^2)^{\frac{1}{2}}\nonumber\\&
			\leq (1-\varepsilon)^{-5}\det(\nabla^2\psi(x))^{\frac{1}{n+2}}, \label{right-control-1}
		\end{align}
  where we have in the second inequality again used that $\frac{1}{1+a}\geq 1-a$ for all $a\in (0, 1)$. 
 In conclusion, \eqref{upper-difference}  and \eqref{right-control-1} give the desired inequality \eqref{control-1} with $\delta_1=\min\{\delta_2, \delta_3\}$ under the  assumption that  the indicatrix of Dupin at $z_x\in\partial\mathsf{epi}(\psi)$ is a sphere. 
\par

If the indicatrix of Dupin at $z_x\in\partial\mathsf{epi}(\psi)$ is an ellipsoid
	 instead of a sphere,  one can find a volume  preserving affine transformation which maps the indicatrix of Dupin at $z_x\in\partial\mathsf{epi}(\psi)$ into a sphere. 
	 As the determinant remains unchanged under a volume preserving affine transformation, 
	 it is easily checked that the desired inequality \eqref{control-1} still holds if the indicatrix of Dupin at $z_x\in\partial\mathsf{epi}(\psi)$ is the boundary of an ellipsoid instead of a sphere, 
	 due to,  for instance,  Proposition \ref{affine-inv-map}. This completes the proof of the desired inequality \eqref{control-1}, which in turn implies the limit \eqref{desied-lim-1}.  
		
\vskip 2mm 
\noindent {\it Case 2: Assume that   $\det(\nabla^2\psi(x))=0$ and $N_{\psi}(z_x)$
	exists uniquely at $x\in\mathbb{R}^n.$ } 
\par
\noindent	
	In this case, \eqref{desied-lim-1}  follows immediately from the following claim: for all $\varepsilon>0$ small enough, there is
	$\delta_4=\delta_4(x,\varepsilon)$ such that for all  $0<\delta\leq\delta_4$, one has 
	\begin{align}
	0\leq\frac{\mathsf{M}_{\delta}(\psi)(x)-\psi(x)}{\delta^{\frac{2}{n+2}}}\leq  b_0\, \varepsilon ^{\frac{n}{n+2}} (1-\varepsilon)^{-3}. \label{control-2}
	\end{align}	 
As $\det(\nabla^2\psi(x))=0$,  the indicatrix of Dupin at $z_x$ is the surface of an elliptic cylinder. This provides again an  approximation of $\partial (\mathsf{epi}(\psi))$ around $z_x$. Indeed, for any $\varepsilon >0 $ small enough, there is an  ellipsoid $\widetilde{\mathcal{E}}$ and a constant $\Xi_{\varepsilon}>0$ such that, (without loss of generality) the lengths of the  first $k$ principal axes of
	$\widetilde{\mathcal{E} }$ are larger than $ \varepsilon^{-1}$, and  for all $0<s\leq\Xi_{\varepsilon},$
	\begin{equation}
		\widetilde{\mathcal{E} } \cap\mathsf{H}^{-}(-N_{\psi}(z_x),s)\subseteq
		\mathsf{epi}(\psi)\cap\mathsf{H}^{-}(-N_{\psi}(z_x),s). \label{det=0-1}
	\end{equation} 
	A proof of this argument can be found in  the proof of 
	\cite[Lemma 23]{SchuettWerner:2004}. There,  the corresponding  statement  was proved for convex bodies but the argument given there 
	 works for unbounded convex set $ \mathsf{epi}(\psi)$ as well.     
\par	 
	  Let $\delta>0$ be small enough and  $0<\Theta_{\delta}<\Xi_{\varepsilon}$ be such that  
	  $$z_{\delta}=(x, \mathsf{M}_{\delta}(\psi)(x))\in\partial\mathsf{M}_{\delta}(\mathsf{epi}(\psi))\cap
	\mathsf{H}^{-}(-N_{\psi}(z_x),\Theta_ \delta).
	$$
	It is enough to prove the claim when $z_{\delta}\in \mathsf{int}(\mathsf{epi}(\psi))$  for all $\delta>0$. If that is not the case,  there is nothing to prove as one  finds a 
	constant $\delta_4>0$ such that $\mathsf{M}_{\delta}(\psi)(x)-\psi(x)=0$ for all $\delta\in (0, \delta_4)$. 
\par
 Let $z_{\delta} \in \mathsf{int}(\mathsf{epi}(\psi))$  for all $\delta>0$ small enough. Assume that  the approximating ellipsoid $\widetilde{\mathcal{E} }$ is a Euclidean ball with radius $\rho$  such that  $\rho\geq  \varepsilon^{-1}$. 
 Let $\widetilde{\Delta}_{h}$ denote the height of a cap of 
 $\widetilde{\mathcal{E}}$ which has volume $\delta$ and let $\widetilde{\Delta}_{\rho}$ be the difference of the radii of
 $\widetilde{\mathcal{E}}$ and $\mathsf{M}_{\delta}(\widetilde{\mathcal{E}}).$
 There exists $\delta_4=\delta_4(x, \varepsilon)$ such that, for all $0<\delta<\delta_4$, 
 \begin{equation}\label{equqat1-22-1}
		\frac{n+4}{n+2}-\varepsilon\leq\frac{\widetilde{\Delta}_{h}}{\widetilde{\Delta}_{\rho}}\leq \frac{n+4}{n+2}+\varepsilon.
	\end{equation} 
	 Moreover, as $\rho\geq  \varepsilon^{-1}$ and by   Lemma \ref{lemma1}, one gets  
	 \begin{align}\label{ineq-12-23-12}
	 \Big(1-\frac{\widetilde{\Delta}_{h}}{2\rho}\Big)^{\frac{n}{n+2}}\widetilde{\Delta}_{h} \leq d_{n+1}\cdot \delta^{\frac{2}{n+2}}\Big(\frac{1}{\rho}\Big)^{\frac{n}{n+2}} \leq  d_{n+1}\cdot  \delta^{\frac{2}{n+2}}  \varepsilon^{\frac{n}{n+2}}. 
	\end{align}    Thus, we can also assume that $\widetilde{\Delta}_{h}<2 \varepsilon$ for all $0<\delta<\delta_4$. 
Applying  again Lemma \ref{lemma2}, \eqref{re1} holds with $\mathcal{B}$
replaced by $\widetilde{\mathcal{E}}$.  That is, for $\delta>0$ small enough,  	 
\begin{equation*} 
\big(\mathsf{M}_{\delta}(\psi)(x)-\psi(x)\big)\left<-N_{\psi}(z_x),e_{n+1}\right>\leq \mathsf{dist}(z_{\delta},  \widetilde{\mathcal{E}}^c) \cdot \Big(1+\frac{2\mathsf{dist}(z_{\delta},  \widetilde{\mathcal{E}}^c)}
{\rho\left<-N_{\psi}(z_x),e_{n+1}\right>^2}\Big)\leq
\mathsf{dist}(z_{\delta},  \widetilde{\mathcal{E}}^c)(1+\varepsilon).
\end{equation*} 
 Note that $\widetilde{\Delta}_{\rho}> \mathsf{dist}(z_{\delta},  \widetilde{\mathcal{E}}^c)$.  
 For $\varepsilon >0$ small enough and  as $\rho\geq  \varepsilon^{-1}$, one has  with \eqref{003},  \eqref{003-d}, \eqref{equqat1-22-1} and \eqref{ineq-12-23-12}
that  for all  $0<\delta\leq\delta_4$, 
\begin{align}
\frac{\mathsf{M}_{\delta}(\psi)(x)-\psi(x)}{c_{n+1} \delta^\frac{2}{n+2}}&\leq b_0\varepsilon ^{\frac{n}{n+2}}
 \Big(1-\frac{\widetilde{\Delta}_{h^{-}}}{2\rho}\Big)^{-\frac{n}{n+2}}  \Big( \frac{d_{n+1}\mathsf{dist}(z_{\delta},  \widetilde{\mathcal{E}}^c) }{c_{n+1} \widetilde{\Delta}_{h} } \Big)
 (1+\varepsilon) \nonumber \\&\leq  b_0\varepsilon ^{\frac{n}{n+2}}  (1-\varepsilon^2)^{-\frac{n}{n+2}}  \Big(1-\frac{(n+2)\varepsilon}{n+4}\Big)^{-1}(1+\varepsilon) 
 \nonumber \\&\leq b_0\varepsilon ^{\frac{n}{n+2}}(1-\varepsilon)^{-\frac{3n+4}{n+2}}\leq
 b_0\varepsilon ^{\frac{n}{n+2}}(1-\varepsilon)^{-3}, \label{right-control-1-222}
\end{align} 
where $b_0=\left<-N_{\psi}(z_x),e_{n+1}\right>^{-1}$.	
This implies the desired inequality  \eqref{control-2}  for all  $0<\delta\leq\delta_4$  under the assumption that  $\widetilde{\mathcal{E} }$ is a Euclidean ball.  
If $\widetilde{\mathcal{E} }$ is an  ellipsoid, then a volume  preserving affine transformation  maps $\widetilde{\mathcal{E} }$ into a Euclidean ball and we conclude as above.

\par
This finishes   the proof of  inequality  \eqref{control-2}, which in turn implies the limit \eqref{desied-lim-1}.   \end{proof}

\bibliography{2}

\end{document}